\documentclass[12pt]{article}

\usepackage[margin=1in]{geometry}
\usepackage[T1]{fontenc}
\usepackage{amsmath}
\usepackage{amsthm}
\usepackage{newtxtext}
\usepackage[subscriptcorrection]{newtxmath}
\usepackage{natbib}
\usepackage{xcolor}
\usepackage{subcaption}
\usepackage{amsfonts,dsfont,mathrsfs}
\usepackage{booktabs,multirow,array}
\usepackage{bm}
\usepackage{graphicx}
\usepackage{bbm}
\usepackage{comment}
\usepackage{url}
\usepackage{xparse}
\usepackage{mathtools}
\usepackage{cleveref}
\usepackage[plain,noend]{algorithm2e}

\makeatletter
\renewcommand{\algocf@captiontext}[2]{#1\algocf@typo. \AlCapFnt{}#2}

\def\@algocf@capt@plain{top}
\renewcommand{\algocf@makecaption}[2]{%
  \addtolength{\hsize}{\algomargin}%
  \sbox\@tempboxa{\algocf@captiontext{#1}{#2}}%
  \ifdim\wd\@tempboxa >\hsize%
    \hskip .5\algomargin%
    \parbox[t]{\hsize}{\algocf@captiontext{#1}{#2}}%
  \else%
    \global\@minipagefalse%
    \hbox to\hsize{\box\@tempboxa}%
  \fi%
  \addtolength{\hsize}{-\algomargin}%
}
\makeatother

\theoremstyle{plain}
\let\amsthmproof\proof
\let\endamsthmproof\endproof
\RenewDocumentEnvironment{proof}{o}
  {\IfNoValueTF{#1}{\amsthmproof}{\amsthmproof[\proofname\ #1]}}
  {\endamsthmproof}

\newtheorem{theorem}{Theorem}
\newtheorem{lemma}{Lemma}

\newtheorem{corollary}{Corollary}
\theoremstyle{remark}
\newtheorem{remark}{Remark}

\newenvironment{keywords}
  {\par\medskip\noindent\textbf{Keywords:}\ }
  {\par\medskip}

\NewDocumentCommand{\evalat}{sO{\big}mm}{
  \IfBooleanTF{#1}
   {\mleft. #3 \mright|_{#4}}
   {#3#2|_{#4}}}

\newcommand{\E}{E}

\renewcommand{\mid}{\ensuremath{\,|\,}}

\usepackage{authblk}
\date{}
\author[1]{Filippo Ascolani}
\author[2]{Mario Beraha}
\author[3]{Stefano Favaro}

\affil[1]{Department of Statistical Science, Duke University}
\affil[2]{Department of Economics, Management, and Statistics, University of Milano--Bicocca}
\affil[3]{Department of Economics and Statistics, University of Torino and Collegio Carlo Alberto}

\title{Asymptotic regimes for maximum likelihood estimation in the Ewens--Pitman model: When the strength parameter matters}

\begin{document}

\maketitle

\begin{abstract}
We study the large sample asymptotic behaviour of the Maximum Likelihood Estimator of the discount and strength parameters $(\alpha,\theta)$ in the Ewens--Pitman model for random partitions, under mild assumptions on the data-generating mechanism. We show that four distinct regimes arise, depending on the limiting behaviour of the frequency spectrum. In particular, in contrast with previous work, we find that $\theta$ may play a crucial role asymptotically. We further show that the existing literature implicitly focuses on only two of these regimes, and we relate this restriction to the constraints imposed by infinite exchangeability. Under the latter, indeed, the number of distinct blocks and the frequency spectrum are necessarily tied by a rigid structural relation. We prove that this lack of flexibility can be overcome through what we call the scaled Ewens--Pitman model, in which $\theta$ is allowed to grow with the sample size $n$. Finally, we provide empirical evidence from real-world data showing that such extensions are needed to capture frequency spectra that fall outside the classical Ewens--Pitman framework.
\end{abstract}


\section{Introduction} \label{sec:intro}
\subsection{Background and motivations}

The Ewens--Pitman (EP) model is a canonical model for random partitions, with applications ranging from genomics \citep{lijoi2007discovering,lijoi2007est,favaro2009speciesvariety} to natural-language processing \citep{goldwater2006types,teh2006hpy}, network analysys \citep{crane2018edge} and forensics \citep{cereda2023learning}. It can be constructed as the marginalization over the space of partitions of a sample generated by a Pitman-Yor process \citep{pitmanyor1997,pitman2006combinatorial}, which occupies a central position in Bayesian nonparametrics \citep{ishwaranjames2003}. Indeed it allows a power-law behaviour in the ranked frequencies, which is observed in many scenarios of scientific interest \citep{clauset2009power}, and leads to a simple and tractable predictive structure. 
See \cite{pitman2006combinatorial} for a review and \eqref{eq:likelihood} for the induced likelihood on the space of partitions.

The EP model depends crucially on two parameters: the discount parameter $\alpha \in (0,1)$ and the strength (or concentration or scale) parameter $\theta > -\alpha$. 
A popular strategy to fix $(\alpha,\theta)$ consists in maximizing the induced likelihood, with an empirical Bayes flavour \citep{lijoi2007est,favaro2009speciesvariety, deblasi2015gibbstype}.
Despite its widespread use, only recent papers studied the asymptotic properties of the resulting estimators $(\hat \alpha_n, \hat \theta_n)$: \cite{koriyama2026asymptotic} provide the asymptotic distribution in the well-specified setting, while \cite{franssen2022empirical} focus on data generated independently from a distribution with regularly varying tails. In both cases it is shown that $\hat \alpha_n$ converges to the ``true" value $\alpha$, while $\hat \theta_n$ is inconsistent for $\theta$ or its distribution does not even concentrate: similar features are shown in \cite{balocchi2026bayesian}. We argue that those results arise from a deeper principle: if the data are generated from an infinitely exchangeable law, then the number of distinct blocks of the partitions and the frequency spectrum are rigidly connected, as formalized later in \eqref{eq:old_assumptions}; see \cite{schweinsberg2010number}. This implies that $\alpha$ alone is sufficient to describe the data generating mechanism, at least asymptotically.

\subsection{Preview of our contributions}

The starting point of the present paper is the empirical observation that, in real data, the Maximum Likelihood Estimator (MLE) often behaves quite differently than prescribed by the available theoretical works.  Figure~\ref{fig:theta-real} reports the value of $\hat \theta_n$ obtained from three large datasets and evaluated along increasing subsets of the sample: in all three cases the fitted values grow systematically as a function of the sample size $n$.
In \Cref{sec:pyp-mle} we give a theoretical justification of this behavior: under mild assumptions on the data generating mechanism, Theorem \ref{thm:mle} gives four different asymptotic regimes for the MLEs. In particular we show that the estimation of $\alpha$, at least asymptotically, depends only on the frequency-of-frequencies spectrum, whereas the asymptotic behavior of $\hat\theta_n$ is entirely driven by the mismatch between $\hat \alpha_n$ and the growth exponent of the number of blocks. Intuitively, this allows the model to break the rigid symmetry described above and to fit more general scenarios: this is empirically confirmed in Section \ref{sec:numeric}.

\begin{figure}
    \centering
    \includegraphics[width=\linewidth]{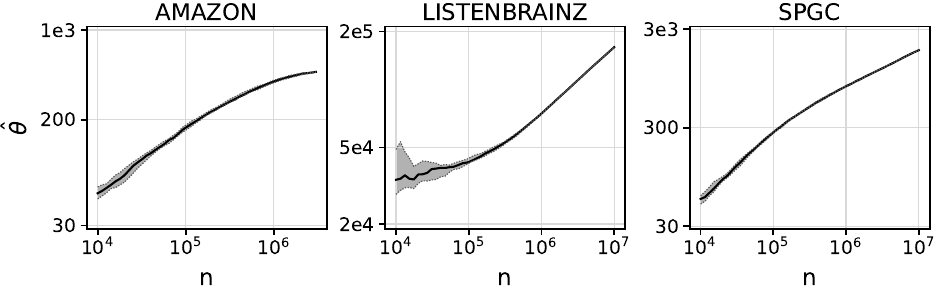}
    \caption{Estimated $\hat \theta_n$ (solid line) and $95\%$ pointwise confidence intervals (shaded area) as a function of the sample size in three different real datasets. See Section \ref{sec:numeric} for more details.}
    \label{fig:theta-real}
\end{figure}

In order to formalize this intuition, in Section \ref{sec:pyp-scaled} we introduce the scaled EP model, where $\theta$ is allowed to grow as $n^\beta$ with $\beta \in (0,1)$. Theorem \ref{thm:scaled-pyp} shows that the asymptotics of the number of blocks depends on both $\alpha$ and $\beta$, while the frequency-of-frequencies spectrum remains invariate: thus, in contrast with \eqref{eq:old_assumptions}, the behaviours of the latter two objects are decoupled and both parameters are relevant even asymptotically.

\section{Maximum likelihood estimation in the Ewens--Pitman model}\label{sec:pyp-mle}
\subsection{The asymptotic behaviour of the Maximum Likelihood Estimators}

Consider a sample $X_{1:n} = (X_1,\dots,X_n)$ featuring $K_n$ distinct values, with multiplicities (or block sizes) recorded in the vector $(N_{1,n},\dots,N_{K_n,n})$ where $N_{i, n}$ is the number of observations in $X_{1:n}$ with label $i$. Denote with $\{M_{r,n}\}_r$ the frequency spectrum, where
\[
M_{r,n}=\sum_{j=1}^{K_n}\mathbf 1\{N_{j,n}=r\}, \qquad r\geq1,
\]
is the number of blocks with size $r$. Assuming that $X_{1:n}$ is an exchangeable vector and only the induced partition is of interest, a common way to model the block sizes is given by the EP exchangeable partition probability function. which specifies the probability of a partition of $\{1,\ldots,n\}$ in $k$ non-empty blocks with multiplicities $(n_1,\ldots,n_k)$ as
\begin{equation}\label{eq:likelihood}
p_{\alpha,\theta}(n_1,\dots,n_k) = \frac{\prod_{i=1}^{k-1}(\theta+i\alpha)}{(\theta+1)_{n-1}} \prod_{j=1}^k(1-\alpha)_{n_j-1}, \qquad \alpha\in(0,1),\ \theta >-\alpha,
\end{equation}
where $(x)_m=\prod_{i=0}^{m-1}(x+i)$ denotes the rising factorial, $n_i > 0$ for $i = 1, \dots, k$ and $\sum_{i = 1}^kn_i = n$. This can be interpreted as a likelihood depending on parameters $(\alpha, \theta)$. When the discount and strength parameters $(\alpha, \theta)$ are fixed, e.g.\, estimated through Maximum Likelihood, it is possible to perform inference on many quantities of interests, such as the number of new blocks in additional samples; see \cite{lijoi2007discovering, favaro2009speciesvariety} and references therein.

In the following we use $\xrightarrow[]{\mathbb P}$ to denote convergence in probability as $n \to \infty$. Moreover, we recall that a function $\ell \, : \, \mathbb{R}_+ \, \to \, \mathbb{R}_+$ is slowly varying  if $\ell(a x)/\ell(x) \to 1$ for every $a > 0$ as $n \to \infty$ \citep{bingham1989regular}. We consider the following assumptions on the data-generating mechanism of $X_{1:n}$.
\begin{itemize}
    \item[\textbf{(A)}] There exist a random variable $C$, a constant $\gamma\in(0,1)$ and a slowly varying function $\ell$ such that
\[
\frac{K_n}{n^{\gamma}\ell(n)}\xrightarrow[]{\mathbb P} C, \quad  0 < C < \infty \text{ a. s.}.
\]
    \item[\textbf{(B)}] There exists a probability mass function $(p_r)_{r\geq1}$ such that $p_1 < 1$ and $\sum_{r\geq2}p_r\log r<\infty$ and
\[
\frac{M_{r,n}}{K_n}\xrightarrow[]{\mathbb P} p_r, \qquad \log n\sum_{r\geq2}\left \lvert \frac{M_{r, n}}{K_n}-p_r \right\rvert\log r \xrightarrow[]{\mathbb P} 0.
\]
\end{itemize}
Assumption~(A) is a Heaps-type law for the number of distinct symbols.  Assumption~(B) requires that the frequency spectrum normalized by the total number of blocks, also called frequency-of-frequencies spectrum, converges suitably to a fixed probability distribution. These requirements are quite mild: for example, if $X_{1:n}$ are generated independently from a fixed probability distribution, then mild conditions on the latter imply convergence of $K_n$ and each $M_{r,n}/K_n$ as above \citep{karlin1967central, gnedin2007notes}. We formally verify assumptions (A) and (B) for our scaled Ewens-Pitman model of Section \ref{sec:pyp-scaled} and the nonexchangeable random partition model of \cite{diBenedetto2021nonexchangeable} in the Appendix (Sections \ref{app:scaled-pyp} and \ref{app:di-benedetto}). Define $\alpha^* \in (0, 1)$ as the unique number satisfying
    \begin{equation}\label{eq:fixed-point}
        \frac{1}{\alpha^*}=\sum_{r\ge2} p_r \sum_{i=1}^{r-1}\frac{1}{i-\alpha^*}.
\end{equation}
We write $c_n = o_p(1)$ if $c_n \xrightarrow[]{\mathbb P} 0$ and $c_n = O_p(1)$ if $(c_n)_n$ is bounded in probability.
Then the next theorem characterizes the asymptotic behaviour of the MLE of $(\alpha, \theta)$.
\begin{theorem}\label{thm:mle}
    Suppose assumptions (A) and (B) hold. If $(\hat \alpha_n, \hat \theta_n)$ maximizes the likelihood $p_{\alpha, \theta}\left(N_{1,n}, \dots, N_{K_n, n} \right)$ in \eqref{eq:likelihood}, then $\hat\alpha_n\xrightarrow[]{\mathbb P}\alpha^*$,
    with $\alpha^*$ as in \eqref{eq:fixed-point}. Moreover, writing $\beta=(\gamma-\alpha^*)/(1-\alpha^*)$, the following alternative regimes hold:
    \[
    \hat{\theta}_n =
    \begin{cases}
        n^\beta\{\alpha^*C\ell(n)\}^{\frac{1}{1-\alpha^*}} (1 + o_p(1)) \quad \text{if } \alpha^* < \gamma\\[0.5ex]
        \{\gamma C\ell(n)\}^{\frac{1}{1-\gamma}}(1+o_p(1)) \quad \text{if } \alpha^* = \gamma \text{ and }\ell(n) \to \infty\\[0.5ex]
        O_p(1) \quad \text{if } \alpha^* = \gamma \text{ and } \ell(n) = O(1)\\[0.5ex]
        -\alpha^*+o_p(1) \quad \text{if } \alpha^* > \gamma
    \end{cases}
    \]
\end{theorem}
\begin{remark}
The proof of Theorem \ref{thm:mle} requires a careful manipulation of the gradient of $\log p_{\alpha, \theta}\left(N_{1,n}, \dots, N_{K_n, n} \right)$ and can be found in the Appendix (Section \ref{app:MLE}). The main step consists in working with
\begin{equation}\label{eq:grad}
\frac{1}{K_n}\frac{\partial }{\partial\alpha}\log p_{\alpha, \theta}\left(N_{1,n}, \dots, N_{K_n, n} \right)  =   \frac{1}{K_n}\sum_{i=1}^{K_n-1}\frac{i}{\theta+i\alpha}  -   \sum_{r\geq2}\frac{M_{r,n}}{K_n}\sum_{m=1}^{r-1}\frac{1}{m-\alpha}.
\end{equation}
In particular, for every $\alpha \in (0,1)$ and suitably small $\theta$, we show that \eqref{eq:grad} converges to
\[
\frac{1}{\alpha}-\sum_{r\ge2} p_r \sum_{i=1}^{r-1}\frac{1}{i-\alpha},
\]
which does not depend on $\theta$. This motivates the definition of $\alpha^*$ in \eqref{eq:fixed-point}.
\end{remark}

Theorem~\ref{thm:mle} is instructive for several reasons. First of all, it shows that asymptotically the estimation of $\alpha$ depends only on $(p_r)_r$, the limit of the frequency-of-frequencies spectrum, and in particular it is not directly related to $K_n$. This is somewhat surprising, since in the EP model $K_n$ behaves as $n^\alpha$: however it is coherent with the fact that, again under the EP model with fixed $(\alpha, \theta)$, it holds that $(M_{r,n}/K_n)_r$ converges almost surely to a distribution $q_\alpha$, which does not depend on $\theta$ \citep[Lemma 1.1]{pitman2006combinatorial}.

Moreover, in contrast with the common view that $\theta$ cannot be properly estimated, we show that it actually plays a relevant role and the asymptotic behaviour of $\hat \theta_n$ varies significantly among the three scenarios. In Section \ref{sec:exchangeability} we will also emphasize that all the previously available results \citep{balocchi2026bayesian, franssen2022empirical, koriyama2026asymptotic} correspond to the case $\alpha^* = \gamma$.

Finally, focusing on the scenario $\alpha^* < \gamma$, which we empirically observe in the datasets we analyze in Section \ref{sec:numeric}, we show in Section \ref{sec:pyp-scaled} that the value of $\beta$ plays the role of correcting the mismatch between the growth of $K_n$ and the asymptotic behaviour of $M_{r, n}/K_n$. See Theorem \ref{thm:scaled-pyp} for more details.


\subsection{The four regimes and connection with infinite exchangeability}\label{sec:exchangeability}

A consequence of all the previously available results on the estimation of the discount and strength parameters $(\alpha, \theta)$ is that $\hat \alpha_n \xrightarrow[]{\mathbb P} \gamma$, which corresponds to two of the four regimes identified in Theorem \ref{thm:mle}: in this section we argue that this is an inevitable consequence of assumption (A) combined with infinite exchangeability of $(X_1, X_2 \dots)$. First of all, the settings of the previous papers imply
\begin{equation}\label{eq:old_assumptions}
\frac{K_n}{n^{\gamma}\ell(n)}\xrightarrow[]{\mathbb P} C \quad \text{and} \quad \frac{M_{r,n}}{K_n}\xrightarrow[]{\mathbb P} \frac{\gamma\Gamma(r-\gamma)}{r!\Gamma(1-\gamma)} =:q_{\gamma}(r),
\end{equation}
where $C$ is a (possibly degenerate) random variable and $q_\gamma$ is the probability distribution of a Sibuya random variable with parameter $\gamma$ \citep{sibuya1979generalized}. Indeed \cite{koriyama2026asymptotic} and \citet[Section $6.1$]{balocchi2026bayesian} consider data generated from the EP model directly, where \eqref{eq:old_assumptions} is known to hold with random $C$ \citep{pitman2006combinatorial}. Instead \cite{franssen2022empirical} and \citet[Section $6.2$]{balocchi2026bayesian} assume $(X_{1:n})_n$ independently sampled from a fixed distribution, which is regularly varying in the sense of Karamata \citep{karlin1967central}: this implies \eqref{eq:old_assumptions} with constant $C$ \citep[Section 7]{gnedin2007notes}.

As a consequence of \citet[Lemma 3]{koriyama2026asymptotic}, if $p_r = q_\gamma(r)$ in Assumption (B) then $\alpha^* = \gamma$: therefore, by Theorem \ref{thm:mle}, \eqref{eq:old_assumptions} implies $\hat \alpha_n = \gamma + o_p(1)$.  Moreover under \eqref{eq:old_assumptions} it is clear that $\alpha$ characterizes the asymptotic behaviour of both the number of blocks and frequency-of-frequencies spectrum almost totally, with $\theta$ only possibly impacting $C$. This intuitively explains why in the papers mentioned above either $\hat \theta_n = O_p(1)$ or it grows very slowly \citep[Theorem 3]{franssen2022empirical}, coherently with the findings in Theorem \ref{thm:mle}: additionally, when the data are generated from the EP model, $\hat \theta_n$ is not consistent for $\theta$. Interestingly \citet[Theorem 2]{schweinsberg2010number} proves that, if $(X_1, X_2, \dots)$ is infinitely exchangeable and the left-hand side of \eqref{eq:old_assumptions} holds, then also the right-hand side must hold. Thus, under infinite exchangeability, the asymptotic behaviour of $K_n$ characterizes also the one of the frequency-of-frequencies spectrum: therefore they can be both captured by the discount parameter alone.

Within this perspective the other two regimes of Theorem \ref{thm:mle} can be considered as being misspecified, namely they correspond to fitting a EP model (which implies infinite exchangeability) to non-infinitely exchangeable data. In the datasets of Figure \ref{fig:theta-real} we find some empirical evidence against \eqref{eq:old_assumptions}, in the sense that the behaviours of $K_n$ and $M_{r,n}/K_n$ cannot always be both matched with a single parameter: see Section \ref{sec:numeric} for more details. In the case $\alpha^* < \gamma$, which we observe in the empirical scenarios, $\alpha$ underestimates the growth rate of the number of blocks in order to provide a better fit for the frequency-of-frequencies spectrum: therefore $\hat \theta_n$ grows polynomially fast to match the "missing rate" for $K_n$. This is formalized in the next section.

\vspace{-0.1cm}
\section{A scaled Ewens--Pitman model}\label{sec:pyp-scaled}

The first regime of Theorem \ref{thm:mle}, where the discount parameter $\alpha$ underestimates the growth rate of $K_n$, leads to an explosion of $\hat\theta_n$, which we also observe empirically. This suggests to consider a modified version of the EP model, where $\theta$ grows with $n$ instead of being fixed. In the next theorem we provide relevant asymptotic properties of this specification, which we call scaled EP model.
\begin{theorem}\label{thm:scaled-pyp}
Let $\lambda>0$, $\alpha\in (0,1)$ and $\beta\in(0,1)$. Consider the EP model with parameters $(\alpha,\theta_n)$, where $\theta_n=\lambda n^\beta$. Then 
\begin{equation}\label{eq:k_a_b}
    \frac{K_n}{n^{\alpha+\beta(1-\alpha)}}\xrightarrow[]{\mathbb P}\frac{\lambda^{1-\alpha}}{\alpha} \quad \text{and} \quad  \frac{M_{r,n}}{K_n}\xrightarrow[]{\mathbb P} q_\alpha(r),
\end{equation}
for every $r \geq 1$, with $q_\alpha$ as in \eqref{eq:old_assumptions}.
\end{theorem}
\begin{remark}
The proof of Theorem \ref{thm:scaled-pyp} relies on asymptotic expansions of the moments of $K_n$ and $M_{r,n}$. They can be found in the Appendix (Section \ref{app:scaled-pyp}) as a corollary of more general results, including the cases $\alpha =0$ and $\beta = 1$. Notice that the setting $\beta = 1$ has been already studied in \cite{BerahaFavaro} for microclustering applications.
\end{remark}

Notice that according to \eqref{eq:k_a_b} the asymptotic behaviour of $M_{r_n}/K_n$ is the same of \eqref{eq:old_assumptions}. However the growth of $K_n$ now depends on both $\alpha$ and $\beta$: therefore, in contrast with \eqref{eq:old_assumptions}, the asymptotics of the number of blocks and the frequency-of-frequencies spectrum have been decoupled. Moreover the values $\alpha = \alpha^*$ and $\beta=(\gamma-\alpha^*)/(1-\alpha^*)$, which correspond to the regime $\alpha^* < \gamma$ of Theorem~\ref{thm:mle}, by  \eqref{eq:k_a_b} imply $K_n = O(n^\gamma)$ coherently with Assumption (A). Therefore this provides a theoretical explanation of what we observe empirically, i.e.\ that $\hat \theta_n$ grows to match the number of blocks.

Since the scaled EP model corresponds to a EP formulation with a choice of $\theta$ that depends on $n$, the law of $(X_{1:n})$ is exchangeable for every $n$ but not infinitely exchangeable: in particular projectivity is lost. As we discussed in Section \ref{sec:exchangeability}, this seems unavoidable in order to model $K_n$ and $M_{r,n}/K_n$ separately. Moreover the limit of $K_n$ in \eqref{eq:k_a_b} is deterministic and not random as in the EP model. This suggests the possibility of constructing asymptotically Gaussian credible intervals for the number of blocks, which we leave for future work.

\vspace{-0.1cm}
\section{Numerical Illustrations}\label{sec:numeric}

\begin{figure}[t]
    \centering
    \includegraphics[width=\linewidth]{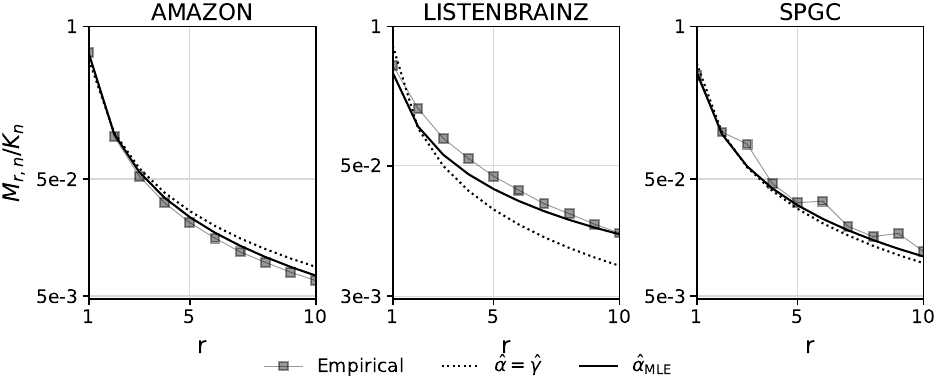}
    \caption{Empirical and estimated frequency-of-frequency spectrum}
    \label{fig:real-data-spectrum}
\end{figure}

\begin{table}[t]
\centering
\begin{tabular}{lccccc}
\toprule
Dataset & $n$ & $K_n$ & $\hat{\gamma}$ & $\hat{\alpha}_{n}$ & $\Delta_{2}$ \\
\midrule
AMAZON & 3,000,000 & 67,786 & 0.520  & 0.487  & 0.0305 \\
LISTENBRAINZ & 10,000,000 & 1,602,241  & 0.527 & 0.36 & 0.1508 \\
SPGC & 10,000,000 & 147,170 & 0.487 & 0.409 & 0.0664 \\
\bottomrule
\end{tabular}
\caption{Summary of the three real datasets of Section \ref{sec:numeric}, where $\Delta_2$ is as in \eqref{eq:squared_difference}.}
\label{tab:pyp-real-data-parameters}
\end{table}

In the Appendix (Section \ref{app:simulation}) we conduct a simulation study where we empirically verify the four regimes of Theorem \ref{thm:mle} and assess the finite sample convergence of the MLE. Here, we illustrate our results on three datasets obtained from public data sources. 
For the Amazon data, we use the Books subset of the Amazon product-review data from the SNAP/McAuley collection 
(\url{https://snap.stanford.edu/data/amazon/productGraph/}); observations are words extracted from the review text, blocks correspond to distinct words, and \(N_{j,n}\) is the number of occurrences of the \(j\)-th word. 
For ListenBrainz, we use the public ListenBrainz full-export dumps 
(\url{https://ftp.musicbrainz.org/pub/musicbrainz/listenbrainz/fullexport/}); observations are listening events, blocks correspond to different MusicBrainz recording identifiers, and \(N_{j,n}\) is the number of listens associated with the \(j\)-th MusicBrainz recording identifier. 
In the Standardized Project Gutenberg Corpus (SPGC), we use the SPGC-2018-07-18 token archive 
(\url{https://zenodo.org/records/2422561}); observations are words, blocks correspond to distinct words, and \(N_{j,n}\) is the number of occurrences of the \(j\)-th word.
See Table~\ref{tab:pyp-real-data-parameters} for the sample size \(n\) and the number of blocks \(K_n\) in the different datasets.

Figure~\ref{fig:theta-real} reports the maximum-likelihood estimate of \(\theta\) along increasing sample sizes.
For each dataset and each sample size, we fit the EP likelihood after randomly permuting the observations 100 times; the solid line is the median over permutations and the shaded region is the corresponding pointwise central band. 
Table~\ref{tab:pyp-real-data-parameters} reports the Heaps exponent \(\hat{\gamma}\), estimated by a log-log regression of \(K_m\) on \(m\) over increasing prefixes, and the  maximum-likelihood estimate \(\hat{\alpha}_{n}\). 
Finally, the last column measures the improvement in the fit of the frequency spectrum obtained by using \(\hat{\alpha}_{n}\) instead of \(\hat{\gamma}\):
\begin{equation}\label{eq:squared_difference}
    \Delta_2 =
    \left\|\hat p - q_{\hat{\gamma}}\right\|_2
    -
    \left\|\hat p - q_{\hat{\alpha}_{n}}\right\|_2,
\end{equation}
where $\hat p = (\hat p_1, \hat p_2, \ldots)$, $ \hat p_r = \frac{M_{r,n}}{K_n}$ is the observed frequency-of-frequency spectrum and $\|\cdot\|_2$ is the standard Euclidean norm. Positive values of \(\Delta_2\) indicate that \(\hat{\alpha}_{n}\) gives a better approximation to the observed frequency-of-frequencies spectrum than \(\hat{\gamma}\).
The joint maximization of the log-likelihood is performed by profiling: for each fixed \(\alpha\), we maximize the log-likelihood over \(\theta\), and then maximize the profiled objective over \(\alpha\). This avoids a poorly conditioned two-dimensional optimization, especially in regimes where the fitted value of \(\theta\) is large. 

The three datasets exhibit the same qualitative pattern. The estimate \(\hat{\theta}_n\) grows with \(n\), and the fitted discount \(\hat{\alpha}_{n}\) is smaller than the Heaps exponent \(\hat{\gamma}\), particularly so in the ListenBrainz dataset. Figure~\ref{fig:real-data-spectrum} explains this behaviour: the empirical frequency spectrum \(M_{r,n}/K_n\) is better approximated by the Sibuya law with parameter \(\hat{\alpha}_{\mathrm{MLE}}\) than by the Sibuya law with parameter \(\hat{\gamma}\), as also reflected by the positive values of \(\Delta_2\) in Table~\ref{tab:pyp-real-data-parameters}. 
In particular, these examples provide empirical evidence that the likelihood uses \(\alpha\) primarily to fit the frequency spectrum, while \(\theta\) increases to compensate for the faster growth of \(K_n\).

\vspace{-0.1cm}
\section{Discussion}

The main finding of this work is that MLEs of the discount and strength parameters $(\alpha,\theta)$ in the Ewens--Pitman model are strongly influenced by the data-generating mechanism. If the observations do not come from an infinitely exchangeable law, and therefore the behaviour of $K_n$ and $M_{r, n}/K_n$ do not match exactly as in \eqref{eq:old_assumptions}, then $(\alpha,\theta)$ are needed to capture relevant characteristics of the data.

The scaled EP model introduced in Section \ref{sec:pyp-scaled} provides a concrete example of decoupling between the growth of $K_n$ and the limiting frequency spectrum. However, inference under the scaled EP is still an open problem: first, the scaled EP likelihood cannot jointly identify $\lambda$ and $\beta$ ad a fixed sample size; secondly, the model does not define a projective family and therefore the predictive distribution is not well-defined. It is then necessary to address the problem of performing principled prediction in a non-infinitely exchangeable setting, to avoid the restrictive conditions in \eqref{eq:old_assumptions}.
A further natural question is whether it is possible to consider a more general class of models, including infinite exchangeability as a particular case: the scaled EP model of Section \ref{sec:pyp-scaled} is an example, recovering the standard EP model with $\beta \to 0$. 
We leave these questions to a future work.

It would be interesting to discuss the rates of convergence of the MLEs of $(\alpha, \theta)$, in the same spirit of \cite{franssen2022empirical, koriyama2026asymptotic} for the well-specified case, and whether they differ depending on the regimes. Similarly we believe that it is possible, but possibly challenging from a technical perspective, to deduce Gaussian central limit theorems for $K_n$ and $M_{r, n}/K_n$ in the setting of Theorem \ref{thm:scaled-pyp}.

\medskip \noindent
\textbf{Declaration of the use of generative AI and AI-assisted technologies.}
GPT 6.2 astra was used to check the correctness of the theoretical results and for the numerical illustrations.

\vspace{-0.5cm}

\paragraph{Acknowledgements} FA was partially supported by the National Institute of Health (grant ID 1R01-GM163225-01).





\bibliographystyle{chicago}
\bibliography{paper-ref}

\clearpage
\appendix
\section*{Appendix}

\renewcommand{\thesection}{\Alph{section}}
\setcounter{section}{0}

\numberwithin{equation}{section}

\numberwithin{lemma}{section}
\renewcommand{\thelemma}{\thesection.\arabic{lemma}}

\numberwithin{proposition}{section}
\renewcommand{\theproposition}{\thesection.\arabic{proposition}}

\numberwithin{theorem}{section}
\renewcommand{\thetheorem}{\thesection.\arabic{theorem}}

\numberwithin{corollary}{section}
\renewcommand{\thecorollary}{\thesection.\arabic{corollary}}

\newcommand{\undalpha}{\underline{\alpha}}

\section{Proof of Theorem \ref{thm:mle}}\label{app:MLE}

We first briefly outline the proof strategy:
\begin{enumerate}
\item Prove that $\alpha^*$ as in \eqref{eq:fixed-point} is well-defined (Lemma \ref{lem:alpha_star_unique}).

\item Prove that $(\hat \alpha_n, \hat \theta_n)$ belongs to the interior of the domain, with probability going to one, and $\hat \theta_n$ is asymptotically smaller than $n^\gamma \ell(n)$ (Lemma \ref{lm:interior}).

\item Prove first that $\hat \alpha_n \xrightarrow[]{\mathbb P} \alpha^*$ (Lemma \ref{lm:convergence_alpha}) and then strengthen the result to $\log n (\hat \alpha_n - \alpha^*) \xrightarrow[]{\mathbb P} 0$ (Corollary \ref{cor:convergence_rate}).

\item Combine the previous points to deduce the four regimes.
\end{enumerate}
Notice that points $1.$ and $2.$ are reminiscent of the results in \citep[Section 3.2]{koriyama2026asymptotic}, under our more general hypotheses.

First of all, we recall a well-known result on slowly-varying function, see e.g.\ Proposition $1.3.6$ in \cite{bingham1989regular}.
\begin{lemma}\label{lm:slowly}
Let $\ell$ be a slowly-varying function. Then $\log (\ell(n))/\log(n) \to 0$ as $n \to \infty$.
\end{lemma}
We will also need the following technical lemma.
\begin{lemma}\label{lm:technical_integral}
For every $K > 0$, $0 \leq m \leq K-1$,  $\alpha > 0$ and $\theta > -m\alpha$ we have that
\[
\frac{1}{\theta +\alpha(K-1)} \leq \sum_{i = m}^{K-1}\frac{1}{\theta+i\alpha}-\frac{1}{\alpha}\log \left(1+\frac{\alpha}{\theta+\alpha m}(K-m-1) \right) \leq \frac{1}{\theta+\alpha m}.
\]
\end{lemma}
\begin{proof}
Since the function $x \, \to \, 1/(\theta+\alpha x)$ is strictly positive and decreasing on $(m, \infty)$, by definition of Riemann integral we have that
\[
\frac{1}{\theta +\alpha(K-1)} \leq \sum_{i = m}^{K-1}\frac{1}{\theta+i\alpha}-\int_m^{K-1}\frac{1}{\theta+\alpha x}\, \text{d} x \leq \frac{1}{\theta+\alpha m},
\]
from which the result follows.
\end{proof}
With $(p_r)_{r\geq 1}$ as in Assumption $(B)$, we define the following function
\begin{equation}\label{eq:definition_G}
G(\alpha)=\frac{1}{\alpha}-\sum_{r\ge2} p_r \sum_{i=1}^{r-1}\frac{1}{i-\alpha}, \quad \alpha \in (0,1).
\end{equation}
The next lemma proves that $\alpha^*$ in \eqref{eq:fixed-point} is well-defined.
\begin{lemma}\label{lem:alpha_star_unique}
Let assumptions $(A)-(B)$ hold and let $G$ be as in \eqref{eq:definition_G}.
Then $G$ is a differentiable and strictly decreasing function on $(0,1)$ with
\[
\lim_{\alpha\downarrow 0}G(\alpha)=+\infty, \quad \lim_{\alpha\uparrow 1}G(\alpha)=-\infty.
\]
As a consequence, there exists a unique $\alpha^*\in(0,1)$ such that $G(\alpha^*)=0$ and therefore $\alpha^*$ in \eqref{eq:fixed-point} is well-defined.
\end{lemma}
\begin{proof}
Fix $\alpha\in(0,1)$. Since $\sum_{j=1}^{m} j^{-1}\le 1+\log m$, for every $r\ge2$ we have
\begin{align*}
\sum_{i=1}^{r-1}\frac{1}{i-\alpha}
&=\frac{1}{1-\alpha}+\sum_{i=2}^{r-1}\frac{1}{i-\alpha}
\\
&\le \frac{1}{1-\alpha}+\sum_{i=2}^{r-1}\frac{1}{i-1}
\le \frac{1}{1-\alpha}+1+\log r.
\end{align*}
This implies that
\[
\sum_{r\ge2} p_r \sum_{i=1}^{r-1}(i-\alpha)^{-1} \leq \frac{2-\alpha}{1-\alpha}+\sum_{r\ge2} p_r \log(r).
\]
Since $\sum_{r\ge2} p_r \log r<\infty$ by assumption, by dominated convergence theorem we conclude that $G(\alpha)$ is a well-defined and differentiable function on $\alpha \in (0,1)$.

Since $p_1 < 1$ by $(B)$, it is immediate to prove that $\lim_{\alpha\downarrow 0}G(\alpha)=+\infty$ and $\lim_{\alpha\uparrow 1}G(\alpha)=-\infty$. Finally
\[
G'(\alpha) = -\frac{1}{\alpha^2}-\sum_{r\ge2} p_r \sum_{i=1}^{r-1}\frac{1}{(i-\alpha)^2} < 0,
\]
as desired.
\end{proof}
Let
\(
\mathcal{D} = \left\{(\alpha, \theta) \, \mid \, \alpha \in [0,1], \,\theta > -\alpha \right\}
\)
be the domain of admissible values for $(\alpha, \theta)$. 
For future reference, we write the log-likelihood
\begin{equation}\label{eq:loglikelihood}
\begin{aligned}
     \ell_n(\alpha, \theta) &:= \log p_{\alpha, \theta}(n_{1}, \ldots, n_k) \\
     & = \sum_{i=1}^{k-1} \log(\theta + i \alpha) - \sum_{i=1}^{n-1} \log(\theta + i) + \sum_{j=1}^k \sum_{m=1}^{n_j - 1} \log(m - \alpha),
\end{aligned}
\end{equation}
and its first derivatives
\begin{equation}\label{eq:first_derivative_theta}
\frac{\partial \ell_n(\alpha, \theta)}{\partial \theta} = \sum_{i = 1}^{K_n-1}\frac{1}{\theta+i\alpha}-\sum_{i = 1}^{n-1}\frac{1}{\theta+i}
\end{equation}
and
\begin{equation}\label{eq:first_derivative_alpha}
\begin{aligned}
\frac{\partial \ell_n(\alpha, \theta)}{\partial \alpha} &= \sum_{i =1}^{K_n-1}\frac{i}{\theta+i\alpha}-\sum_{j =1}^{K_n}\sum_{m = 1}^{n_j-1}\frac{1}{m-\alpha}\\
&= \sum_{i = 1}^{K_n-1}\frac{i}{\theta+i\alpha}-\sum_{r= 2}^\infty M_{r,n}\sum_{i = 1}^{r-1}\frac{1}{i-\alpha}.
\end{aligned}
\end{equation}
We collect some basic properties of $(\hat \alpha_n, \hat \theta_n)$ in the following lemma.

\begin{lemma}\label{lm:interior}
Let Assumptions $(A)-(B)$ hold and let $(\hat\alpha_n,\hat\theta_n)$ be a local maximizer of $p_{\alpha, \theta}$. Then with probability $1+o_p(1)$ we have that $(\hat\alpha_n,\hat\theta_n)$ belongs to the interior of $\mathcal{D}$ for $n$ large enough and
\begin{equation}\label{eq:condition1}
\hat\theta_n \leq dC\frac{n^\gamma \ell(n)}{\log(n)}
\end{equation}
for some $d > 0$ and
\begin{equation}\label{eq:condition2}
K_n = \sum_{i = 0}^{n-1}\frac{\hat\theta_n}{\hat\theta_n+i}+\sum_{i = 0}^{K_n-1}\frac{\hat\alpha_n i}{\hat\theta_n+i\hat\alpha_n}
\end{equation}
and
\begin{equation}\label{eq:condition3}
\frac{1}{K_n}\sum_{i = 0}^{K_n-1}\frac{ i}{\hat\theta_n+i\hat\alpha_n} - \sum_{r\geq 2}\frac{M_{r,n}}{K_n}\sum_{i = 1}^{r-1}\frac{1}{i-\hat\alpha_n} = 0,
\end{equation}
in probability as $n \to \infty$.
\end{lemma}
\begin{proof}
Let $C_n = \frac{K_n}{n^\gamma\ell(n)}$. By assumption (A), $C_n \to C$ in probability. From now on, work on the event $1 < K_n < n$, whose probability tends to one by assumption $(A)$.

If $(\hat\alpha_n,\hat\theta_n)$ belongs to the interior of $\mathcal{D}$ then \eqref{eq:first_derivative_theta} and \eqref{eq:first_derivative_alpha} must be equal to zero at $(\hat \alpha_n, \hat\theta_n)$. The equation $K_n^{-1}\frac{\partial \ell_n(\alpha, \theta)}{\partial \alpha} = 0$ corresponds exactly to \eqref{eq:condition3}. Moreover
\[
0 = \hat\theta_n \frac{\partial \ell_n(\hat\alpha_n, \hat\theta_n)}{\partial \theta} = K_n - \sum_{i = 0}^{K_n-1}\frac{\hat\alpha_n i}{\hat\theta_n+i\hat\alpha_n}-\sum_{i = 0}^{n-1}\frac{\hat\theta_n}{\hat\theta_n+i},
\]
which corresponds exactly to \eqref{eq:condition2}. We are then left to show that $(\hat\alpha_n, \hat\theta_n)$ belongs to the interior of $\mathcal{D}$ and that \eqref{eq:condition1} holds.

It is immediate to see that $\ell_n \to -\infty$ if $\theta \to +\infty$ or $\alpha \to 1$ or $\theta \to -\alpha$. The last boundary case is $\hat{\alpha}_n = 0$, for which $\hat\theta_n$ must satisfy
\[
K_n = \sum_{i = 0}^{n-1}\frac{\hat\theta_n}{\hat\theta_n+i},
\]
which corresponds to \eqref{eq:condition2} with $\hat\alpha_n = 0$. In order to prove \eqref{eq:condition1}, if $\hat \theta_n > 0$ we can always conclude by Lemma \ref{lm:technical_integral} (with $K=n$, $m=0$, $\alpha=1$, and $\theta=\hat\theta_n$) that
\[
K_n  = \sum_{i = 0}^{n-1}\frac{\hat\theta_n}{\hat\theta_n+i}+\sum_{i = 0}^{K_n-1}\frac{\hat\alpha_n i}{\hat\theta_n+i\hat\alpha_n} \geq \sum_{i = 0}^{n-1}\frac{\hat\theta_n}{\hat\theta_n+i} \geq \hat\theta_n\int_0^{n-1}\frac{1}{\hat\theta_n+x}\, d x = \hat\theta_n \log \left(1+\frac{n-1}{\hat\theta_n} \right),
\]
which implies that
\[
\frac{K_n}{n} > \frac{\hat\theta_n}{n} \log \left(1+\frac{n}{\hat\theta_n}\right) + O\left( 1/n\right).
\]

Work now on the event $\left\{ \hat\theta_n > dCn^\gamma \ell(n)/\log(n)\right\}$ for a fixed $d>1/(1-\gamma)$. Since the function $h(x) = x^{-1}\log(1+x)$ is decreasing in $x$, by the previous display and assumption $(A)$ we have that
\[
\frac{C_n}{C} \geq \frac{d}{\log(n)}\log\left( 1+\frac{n^{1-\gamma}\log(n)}{dC\,\ell(n)}\right)+o_p\left(1\right).
\]
By Lemma \ref{lm:slowly}, and since \(\log C/\log n\to0\) in probability, the right-hand side converges in probability to \(d(1-\gamma)\), which contradicts \(C_n/C\to1\) in probability. Then \eqref{eq:condition1} follows with $d > 1/(1-\gamma)$.

To exclude the remaining boundary case, suppose by contradiction that $\hat \alpha_n = 0$. We notice that
\[
\begin{aligned}
\frac{1}{K_n}\frac{\partial \ell_n(0, \hat\theta_n)}{\partial \alpha} &= \frac{1}{\hat\theta_n K_n}\sum_{i = 0}^{K_n-1}i-\sum_{r= 2}^\infty \frac{M_{r,n}}{K_n}\sum_{i = 1}^{r-1}\frac{1}{i} \geq \frac{K_n-1}{2\hat\theta_n}-\sum_{r= 2}^\infty \frac{M_{r,n}}{K_n}\log(r) -1.
\end{aligned}
\]
By \eqref{eq:condition1} we have that
\[
\frac{1}{K_n}\frac{\partial \ell_n(0, \hat\theta_n)}{\partial \alpha} \geq  \frac{C_n}{2dC}\log(n)-\sum_{r= 2}^\infty p_r\log(r) + O_p(1).
\]
The second term is \(O_p(1)\) by Assumption \((B)\), while \(C_n/C\to1\) in probability. Hence $\frac{\partial \ell_n(0, \hat\theta_n)}{\partial \alpha} > 0$ with probability tending to one and therefore $(\hat \alpha_n, \hat \theta_n)$ must belong to the interior of $\mathcal{D}$.
\end{proof}

Define now
\begin{equation}\label{eq:definition_Gn}
G_n(\alpha, \theta): = \frac{1}{K_n}\frac{\partial \ell_n(\alpha, \theta)}{\partial \alpha}=\frac{1}{K_n}\sum_{i = 0}^{K_n-1}\frac{i}{\theta+i\alpha}-\sum_{r= 2}^\infty \frac{M_{r,n}}{K_n}\sum_{i = 1}^{r-1}\frac{1}{i-\alpha}.
\end{equation}

For simplicity we also define $G_n(\alpha,\theta) = 1$ when $(\alpha, \theta) \not \in \mathcal{D}$, where $1$ is an arbitrary value. Notice that \eqref{eq:condition3} corresponds to $G_n(\hat\alpha_n, \hat \theta_n) = 0$. For ease of notation, let $\underline{\alpha}_n = \max \{0, - \hat \theta_n \}$ and notice that by Lemma \ref{lm:interior} we have that $\hat \alpha_n > \underline{\alpha}_n$. The next lemma proves convergence of $\hat \alpha_n$ and its proof is inspired by the one of Lemma $5.10$ in \cite{van2000asymptotic}.
\begin{lemma}\label{lm:convergence_alpha}
Let $(A)-(B)$ hold and let $G$ and $G_n$ be as in \eqref{eq:definition_G} and \eqref{eq:definition_Gn}, respectively. Then we have that
\begin{enumerate}
\item $\mathbb{P}\left(\hat \theta_n +\hat \alpha_n \geq \frac{1}{2 \log n}, \hat \theta_n \geq -\alpha^* \right) \to 1$, as $n \to \infty$.
\item  $\alpha \, \to \, G_n(\alpha, \hat \theta_n)$ is a continuous and strictly decreasing function on $(\undalpha_n,1)$. Moreover
\[
\mathbb{P} \left(\lim_{\alpha\downarrow \undalpha_n}G_n(\alpha, \hat \theta_n) > 0, \quad \lim_{\alpha\uparrow 1}G_n(\alpha, \hat \theta_n) < 0\right) \to 1,
\]
as $n \to \infty$.

\item $\left\lvert G_n(\alpha, \hat \theta_n)-G(\alpha)\right\rvert \xrightarrow[]{\mathbb P} 0$ for every $\alpha \in (\alpha^*,1)$ and $\left\lvert G_n(\alpha, \hat \theta_n)-G(\alpha)\right\rvert\mathbbm{1}_{\{\hat \theta_n \geq 0\}} \xrightarrow[]{\mathbb P} 0$ for every $\alpha \in (0,1)$.

\item $\hat \alpha_n \xrightarrow[]{\mathbb P} \alpha^*$.
\end{enumerate}
\end{lemma}
\begin{proof}
As a first technical step, we prove that
\begin{equation}\label{eq:lemma_toshow2}
\log n\left\lvert \sum_{r= 2}^\infty \frac{M_{r,n}}{K_n}\sum_{i = 1}^{r-1}\frac{1}{i-\alpha}-\sum_{r= 2}^\infty p_r\sum_{i = 1}^{r-1}\frac{1}{i-\alpha}\right\rvert = o_p(1),
\end{equation}
for every $\alpha \in (0,1)$. Indeed, notice that
\[
\sum_{i=1}^{r-1}\frac{1}{i-\alpha}\le \frac{1}{1-\alpha}+\sum_{i=2}^{r-1}\frac{1}{i-1}\le \frac{1}{1-\alpha}+1+\log r,
\]
for every $r \geq 2$. Therefore
\begin{equation}\label{eq:inequality2}
\begin{aligned}
\log n\biggl\lvert \sum_{r= 2}^\infty \frac{M_{r,n}}{K_n}\sum_{i = 1}^{r-1}\frac{1}{i-\alpha}&-\sum_{r= 2}^\infty p_r\sum_{i = 1}^{r-1}\frac{1}{i-\alpha}\biggr\rvert \\
&\leq\log n\sum_{r\ge2}\left|\frac{M_{r,n}}{K_n}-p_r\right|\left(\frac{1}{1-\alpha}+1+\log r\right) \xrightarrow[]{\mathbb P} 0,
\end{aligned}
\end{equation}
as $n \to \infty$ by assumption $(B)$.

As regards point $1.$, if $\hat \theta_n > 0$ by \eqref{eq:first_derivative_theta} we have
\[
\frac{1}{\hat \theta_n +\hat \alpha_n} \leq \sum_{i = 1}^{K_n-1}\frac{1}{\hat \theta_n +\hat \alpha_ni} \leq \sum_{i = 1}^{n-1}\frac{1}{i} \leq 2\log n,
\]
with $n$ large enough, and the result follows. Consider now $\hat \theta_n \leq 0$ and notice that combining \eqref{eq:condition2} and \eqref{eq:condition3} we deduce that
\begin{equation}\label{eq:score_equality_Lemma5}
\sum_{r\geq 2}\frac{M_{r,n}}{K_n}\sum_{i = 1}^{r-1}\frac{1}{i-\hat\alpha_n} = \frac{K_n-1}{\hat \alpha_n K_n} - \frac{\hat \theta_n}{\hat \alpha_n K_n}\sum_{i = 1}^{n-1}\frac{1}{\hat\theta_n+i}.
\end{equation}
Combining this with $\hat \theta_n > -\hat \alpha_n$,
\[
\sum_{i = 1}^{n-1}\frac{1}{\hat\theta_n+i} \leq \frac{1}{1-\hat \alpha_n} +\sum_{i = 2}^{n-1}\frac{1}{i-1} \leq \frac{1}{1-\hat \alpha_n} + 1 + \log n
\]
and
\[
\sum_{r\geq 2}\frac{M_{r,n}}{K_n}\sum_{i = 1}^{r-1}\frac{1}{i-\hat\alpha_n} \geq \frac{1-M_{1,n}/K_n}{1-\hat \alpha_n},
\]
we deduce that
\[
1-\frac{M_{1,n}}{K_n} \leq \frac{1-\hat \alpha_n}{\hat \alpha_n}+\frac{2+\log n}{K_n}.
\]
Since the left-hand side converges to a strictly positive number by $(B)$, we conclude that there exists $\epsilon > 0$ such that
\begin{equation}\label{eq:bounded_away_1}
\mathbb{P}(\hat \theta_n \leq -1+ \epsilon) \leq \mathbb{P}(\hat \theta_n \leq 0, \hat \alpha_n \geq 1- \epsilon) \to 0,
\end{equation}
as $n \to \infty$. Then, on the event $\{\hat \theta_n > -1+\epsilon\}$,  \eqref{eq:first_derivative_theta} gives that
\[
\frac{1}{\hat \theta_n + \hat \alpha_n} \leq \sum_{i = 1}^{K_n-1}\frac{1}{\hat \theta_n + \hat \alpha_n i} = \sum_{i = 1}^{n-1}\frac{1}{\hat \theta_n + i} \leq \frac{1}{\epsilon} + 1 + \log n,
\]
which implies
\begin{equation}\label{eq:lm5point11}
\mathbb{P}\left(\hat \theta_n +\hat \alpha_n \geq \frac{1}{2 \log n} \right) \to 1,
\end{equation}
as $n \to \infty$. Consider now $\hat \theta_n < - \alpha^*$: then by monotonicity in $\alpha$ and \eqref{eq:score_equality_Lemma5} we get that 
\[
\sum_{r\geq 2}\frac{M_{r,n}}{K_n}\sum_{i = 1}^{r-1}\frac{1}{i-\alpha^*} \leq  \frac{K_n-1}{\hat \alpha_n K_n} - \frac{\hat \theta_n}{\hat \alpha_n K_n}\sum_{i = 1}^{n-1}\frac{1}{\hat\theta_n+i}.
\]
Focusing on the latest sum, observe that on $\{\hat \theta_n > -1 + \varepsilon\}$ we have
\[
    \sum_{i = 1}^{n-1}\frac{1}{\hat\theta_n+i} = \frac{1}{\hat\theta_n+1} + \sum_{i = 2}^{n-1}\frac{1}{\hat\theta_n+i} \le \frac{1}{\varepsilon} + 1 + \log(n),
\]
which implies
\[
\sum_{r\geq 2}\frac{M_{r,n}}{K_n}\sum_{i = 1}^{r-1}\frac{1}{i-\alpha^*} \leq \frac{1}{\hat \alpha_n}+\frac{\epsilon^{-1}+1+\log n}{K_n},
\]
and thus
\[
\begin{aligned}
\hat \alpha_n - \alpha^* &= \hat \alpha_n \alpha^*\left(\frac{1}{\alpha^*}-\frac{1}{\hat \alpha_n} \right)\\
&\leq \left\lvert \sum_{r= 2}^\infty \frac{M_{r,n}}{K_n}\sum_{i = 1}^{r-1}\frac{1}{i-\alpha^*}-\frac{1}{\alpha^*}\right\rvert+\frac{\epsilon^{-1}+1+\log n}{K_n},
\end{aligned}
\]
with probability going to one. By \eqref{eq:lemma_toshow2} and $G(\alpha^*) = 0$ we have that the right-hand side is $o_p(1/\log n)$. However, since $\hat \alpha_n +\hat \theta_n < \hat \alpha_n - \alpha^*$ this contradicts \eqref{eq:lm5point11} and point 1. follows.

As regards point $2.$, it is easy to see that $G_n(\alpha, \hat \theta_n)$ is continuous and decreasing at every $\alpha \in (\undalpha_n,1)$. Moreover
\[
G_n(\alpha, \hat \theta_n) \leq \frac{K_n-1}{2\left(\hat \theta_n+\alpha\right)}-\frac{1}{1-\alpha}\sum_{r \geq 2}\frac{M_{r,n}}{K_n},
\]
which implies $\lim_{\alpha\uparrow 1}G_n(\alpha, \hat \theta_n) = -\infty$. Moreover, if $\hat \theta_n \leq 0$ and $K_n > 1$ we have
\[
G_n(\alpha, \hat \theta_n) \geq \frac{1}{K_n(\hat \theta_n + \alpha)} -\sum_{r= 2}^\infty \frac{M_{r,n}}{K_n}\sum_{i = 1}^{r-1}\frac{1}{i-\alpha},
\]
from which we deduce $\lim_{\alpha\downarrow \undalpha_n}G_n(\alpha, \hat \theta_n) = +\infty$. If instead $\hat \theta_n > 0$, we deduce that
\[
\begin{aligned}
G_n(0, \hat \theta_n) = \frac{1}{\hat\theta_n K_n}\sum_{i = 0}^{K_n-1}i-\sum_{r= 2}^\infty \frac{M_{r,n}}{K_n}\sum_{i = 1}^{r-1}\frac{1}{i} \geq \frac{K_n-1}{2\hat\theta_n}-\sum_{r= 2}^\infty \frac{M_{r,n}}{K_n}\log(r) -1,
\end{aligned}
\]
and by \eqref{eq:condition1} and Assumption $(B)$ we have that
\[
G_n(0, \hat \theta_n) \geq  \frac{C_n}{2dC}\log(n)-\sum_{r= 2}^\infty p_r\log(r) + O_p(1),
\]
which implies that $G_n(0, \hat \theta_n) > 0$ for $n$ large enough. Thus,
\[
\begin{aligned}
\mathbb{P} \biggl(\lim_{\alpha\downarrow \undalpha_n}G_n(\alpha, \hat \theta_n) > 0,& \quad \lim_{\alpha\uparrow 1}G_n(\alpha, \hat \theta_n) < 0\biggr)\\
&\geq \mathbb{P} \left(1 < K_n < n, \frac{K_n-1}{2|\hat\theta_n|}>1+\sum_{r= 2}^\infty \frac{M_{r,n}}{K_n}\log(r)\right) \to 1,
\end{aligned}
\]
as $n \to \infty$, as desired.

As regards point $3.$, notice that
\[
\frac{1}{K_n}\sum_{i = 0}^{K_n-1}\frac{i}{\hat\theta_n+i\alpha}= \frac{1}{\alpha}\left(1-\frac{\hat\theta_n}{K_n}\sum_{i = 0}^{K_n-1}\frac{1}{\hat\theta_n+i\alpha} \right)
\]
and
\[
\begin{aligned}
\sum_{i = 0}^{K_n-1}\frac{1}{\hat\theta_n+i\alpha} \leq  \frac{1}{\hat \theta_n}+\frac{1}{\alpha}\log \left(1+\frac{\alpha(K_n-2)}{\hat\theta_n+\alpha} \right) +\frac{1}{\hat\theta_n+\alpha},
\end{aligned}
\]
by Lemma \ref{lm:technical_integral}. Combining the above two formulas, for $\alpha > \undalpha_n$ we have that
\begin{equation}\label{eq:inequality1}
\begin{aligned}
\left\lvert\frac{1}{K_n}\sum_{i = 0}^{K_n-1} \frac{i}{\hat\theta_n+i\alpha}-\frac{1}{\alpha}\right\rvert 
& = \left\lvert - \frac{1}{\alpha K_n} - \frac{\hat \theta_n}{\alpha K_n} \sum_{i=1}^{K_n - 1} \frac{1}{\hat \theta_n + i \alpha} \right\rvert \\
& \qquad \leq \frac{1}{ \alpha K_n}+\frac{1}{\alpha^2}\frac{|\hat\theta_n|}{K_n}\log \left(1+\frac{\alpha(K_n-2)}{\hat\theta_n+\alpha} \right) +\frac{|\hat \theta_n|}{\alpha(\hat \theta_n+\alpha) K_n}.
\end{aligned}
\end{equation}
By definition, for $\alpha > \undalpha_n$ we have that
\begin{equation}\label{eq:diff_G_G_n}
G_n(\alpha, \hat \theta_n)-G(\alpha) = \left[ \frac{1}{K_n}\sum_{i = 0}^{K_n-1}\frac{i}{\hat\theta_n+i\alpha}-\frac{1}{\alpha}\right] -\sum_{r\geq 2}\left(\frac{M_{r,n}}{K_n}-p_r\right)\sum_{i = 1}^{r-1}\frac{1}{i-\alpha}, 
\end{equation}
where the term outside the square brackets is $o_p(1)$ by \eqref{eq:lemma_toshow2}. Consider first $\hat \theta_n \geq 0$ with $\alpha \in (0,1)$. By \eqref{eq:inequality1} we have that
\[
\left\lvert\frac{1}{K_n}\sum_{i = 0}^{K_n-1}\frac{i}{\hat\theta_n+i\alpha}-\frac{1}{\alpha}\right\rvert \leq \frac{2}{\alpha K_n}+\frac{\hat\theta_n}{\alpha^2K_n}\log \left(1+\frac{\alpha K_n}{\hat\theta_n} \right).
\]
By \eqref{eq:condition1} and the fact that $x \, \to \, x\log(1+\alpha/x)$ is increasing, with probability tending to one the right-hand side above is at most
\[
\frac{2}{\alpha K_n}+\frac{2d}{\alpha^2\log n}\log \left(1+\frac{\alpha \log n}{2d} \right) \xrightarrow[]{\mathbb P} 0,
\]
which by \eqref{eq:diff_G_G_n} implies
\begin{equation}\label{eq:limit_G_n1}
\left\lvert G_n(\alpha, \hat \theta_n)-G(\alpha)\right\rvert\mathbbm{1}_{\{\hat \theta_n \geq 0\}} \xrightarrow[]{\mathbb P} 0.
\end{equation}
Now fix $\alpha \in (\alpha^*, 1)$ and consider $\hat \theta_n < 0$, which by point $1.$ implies $\hat \theta_n \geq -\alpha^*$ with probability going to one. On this event by \eqref{eq:inequality1} we have that
\[
\left\lvert\frac{1}{K_n}\sum_{i = 0}^{K_n-1}\frac{i}{\hat\theta_n+i\alpha}-\frac{1}{\alpha}\right\rvert \leq  \frac{1}{ \alpha K_n}+\frac{1}{\alpha^2}\frac{\alpha^*}{K_n}\log \left(1+\frac{\alpha(K_n-2)}{\alpha-\alpha^*} \right) +\frac{\alpha^*}{\alpha(\alpha-\alpha^*) K_n}\xrightarrow[]{\mathbb P} 0,
\]
since $K_n \to \infty$ in probability by $(A)$. This then implies
\[
\left\lvert G_n(\alpha, \hat \theta_n)-G(\alpha)\right\rvert\mathbbm{1}_{\{\hat \theta_n < 0\}} \xrightarrow[]{\mathbb P} 0,
\]
which, combined with \eqref{eq:limit_G_n1}, finally gives the result. 

As regards point $4.$, fix $0 < \epsilon < \min \{\alpha^*, 1- \alpha^*\}$ and denote
\[
A_n = \left\{\lim_{\alpha\downarrow \undalpha_n}G_n(\alpha, \hat \theta_n) > 0, \quad \lim_{\alpha\uparrow 1}G_n(\alpha, \hat \theta_n) < 0, \quad \hat \theta_n \geq -\alpha^* \right\},
\]
where by points $1.$ and $2.$ we have that $\mathbb{P}(A^c_n) = o_p(1)$. Then by continuity of $G_n$ we have that
\[
\begin{aligned}
    \mathbb{P}&(\hat \alpha_n > \alpha^* +\epsilon)  \leq \mathbb{P}(A_n^c) + \mathbb{P}\left(G_n(\alpha^*+\epsilon, \hat \theta_n) >0, A_n \right)\\
    & \leq \mathbb{P}(A_n^c) + \mathbb{P}\left(G(\alpha^* +\epsilon) +G_n(\alpha^*+\epsilon, \hat \theta_n)\mathbbm{1}_{\{\alpha^*+\epsilon > \undalpha_n\}} - G(\alpha^* +\epsilon)\mathbbm{1}_{\{\alpha^*+\epsilon > \undalpha_n\}} > 0 \right) \to 0,
\end{aligned}
\]
as $n \to \infty$, by point $3.$ and Lemma \ref{lem:alpha_star_unique}. Instead, if $\hat \theta_n \leq 0$ and $\hat \alpha_n < \alpha^*-\epsilon$ then by point 2. we have that
\[
G_n(\hat \alpha_n, \hat \theta_n) > G_n(\alpha^*-\epsilon, \hat \theta_n) \geq G_n(\alpha^*-\epsilon, 0) \xrightarrow[]{\mathbb P} G(\alpha^*-\epsilon) > 0,
\]
where the limit holds by \eqref{eq:lemma_toshow2}. Therefore
\[
\mathbb{P}(\hat \alpha_n < \alpha^*-\epsilon, \hat \theta_n \leq 0) \leq \mathbb{P}\left(G_n(\alpha^*-\epsilon, 0) < 0\right) +o_p(1) \to 0.
\]
Finally we have
\[
\begin{aligned}
    \mathbb{P}&(\hat \alpha_n < \alpha^* -\epsilon, \hat \theta_n > 0)  \leq \mathbb{P}(A_n^c) + \mathbb{P}\left(G_n(\alpha^*-\epsilon, \hat \theta_n) <0, \hat \theta_n > 0 \right)\\
    & \leq \mathbb{P}(A_n^c) + \mathbb{P}\left(G(\alpha^* -\epsilon) +G_n(\alpha^*-\epsilon, \hat \theta_n)\mathbbm{1}_{\{\hat \theta_n > 0\}} - G(\alpha^* -\epsilon)\mathbbm{1}_{\{\hat \theta_n > 0\}} < 0 \right) \to 0,
\end{aligned}
\]
as $n \to \infty$, again by point $3.$ and Lemma \ref{lem:alpha_star_unique}.
\end{proof}

\begin{corollary}\label{cor:convergence_rate}
Let $(A)-(B)$ hold. Then we have that
\begin{enumerate}
  \item for every $\delta > 0$ with probability $1 + o_p(1)$ we have that 
 $\hat \theta_n \leq n^{\kappa + \delta}$, 
where
\[
\kappa = 
\begin{cases}
   \frac{\gamma - \alpha^*}{1-\alpha^*} \quad \text{if } \alpha^* \leq \gamma \\
   &\\
   0 \quad \text{else} 
\end{cases}
\]

\item $\log (n) \left(\hat \alpha_n-\alpha^*\right) \xrightarrow[]{\mathbb P} 0$.
\end{enumerate}
\end{corollary}
\begin{proof}
As regards point $1.$, fix $\delta > 0$ and work on the event $\{\hat \theta_n > n^{\kappa + \delta}\}$. By Lemma \ref{lm:interior} with probability going to one we have that
\[
\sum_{i = 0}^{K_n-1}\frac{1}{\hat{\theta}_n+i\hat \alpha_n}-\sum_{i = 0}^{n-1}\frac{1}{\hat\theta_n+i} = 0.
\]
Since by Lemma \ref{lm:technical_integral} we have that
\[
\sum_{i = 0}^{K_n-1}\frac{1}{\hat\theta_n+i\hat \alpha_n} \leq  \frac{1}{\hat \theta_n}+\frac{1}{\hat \alpha_n}\log \left(1+\frac{\hat \alpha_n(K_n-2)}{\hat\theta_n+\hat \alpha_n} \right) +\frac{1}{\hat\theta_n+\hat \alpha_n}
\]
and 
\[
\sum_{i = 0}^{n-1}\frac{1}{\hat\theta_n+i} \geq \log \left(1+\frac{n-1}{\hat \theta_n} \right),
\]
we get that
\[
\frac{1}{\hat \theta_n}+\frac{1}{\hat \alpha_n}\log \left(1+\frac{\hat \alpha_n(K_n-2)}{\hat\theta_n+\hat \alpha_n} \right) +\frac{1}{\hat\theta_n+\hat \alpha_n} \geq \log \left(1+\frac{n-1}{\hat \theta_n} \right).
\]
Combined with \eqref{eq:condition1}, from the above formula we deduce that
\[
\log n -\log \hat \theta_n \leq \frac{1}{\hat \alpha_n} \left(\log K_n - \log \hat \theta_n \right) + O_p(1),
\]
which means
\[
(\kappa + \delta) \log n < \log \hat \theta_n \leq \frac{\log K_n -\hat \alpha_n \log n}{1-\hat \alpha_n} +O_p(1).
\]
By Assumption $(A)$ and Lemma \ref{lm:convergence_alpha} (point 4.) we can conclude that this is a contradiction.

 As regards point $2.$, by a Taylor expansion we have that
\[
\begin{aligned}
0 &= G_n(\hat\alpha_n, \hat \theta_n) - G(\alpha^*) = G_n(\hat\alpha_n, \hat \theta_n) - G(\hat \alpha_n) + G(\hat \alpha_n) - G(\alpha^*)\\
& = G_n(\hat\alpha_n, \hat \theta_n) - G(\hat \alpha_n) + (\hat \alpha_n - \alpha^*)G'(\bar{\alpha}),
\end{aligned}
\]
with $\bar{\alpha}$ in an arbitrarily small neighborhood of $\alpha^*$, with probability $1 + o_p(1)$ by Lemma \ref{lm:convergence_alpha} (point 4.). Since $G'(\alpha)$ is bounded away from zero for $\alpha \in (0,1)$, by Lemma  \ref{lem:alpha_star_unique}, we deduce that
\[
\hat \alpha_n - \alpha^* = \frac{G(\hat \alpha_n)-G_n(\hat\alpha_n, \hat \theta_n)}{G'(\bar{\alpha})},
\]
with probability $1 + o_p(1)$, and we are left to prove
\[
\log n \left[G(\hat \alpha_n)-G_n(\hat\alpha_n, \hat \theta_n) \right] \xrightarrow[]{\mathbb P} 0.
\]
By \eqref{eq:diff_G_G_n} this follows by 
\[
\log n\biggl\lvert \sum_{r= 2}^\infty \frac{M_{r,n}}{K_n}\sum_{i = 1}^{r-1}\frac{1}{i-\hat\alpha_n}-\sum_{r= 2}^\infty p_r\sum_{i = 1}^{r-1}\frac{1}{i-\hat\alpha_n}\biggr\rvert \xrightarrow[]{\mathbb P} 0,
\]
since $\hat \alpha_n \leq (1+\alpha^*)/2$ with probability going to one by Lemma \ref{lm:convergence_alpha} (point 4.), and by
\[
 \log n\left\lvert\frac{1}{K_n}\sum_{i = 0}^{K_n-1}\frac{i}{\hat\theta_n+i\hat \alpha_n}-\frac{1}{\hat\alpha_n}\right\rvert = o_p(1),
\]
which we now prove. Indeed, by \eqref{eq:inequality1} we have that
\[
\begin{aligned}
\log n&\left\lvert\frac{1}{K_n}\sum_{i = 0}^{K_n-1}\frac{i}{\hat\theta_n+i\hat \alpha_n}-\frac{1}{\hat\alpha_n}\right\rvert\\
&\leq \frac{\log n}{\hat \alpha_nK_n}+\frac{\log n}{\hat \alpha_n^2}\frac{|\hat\theta_n|}{K_n}\log \left(1+\frac{\hat \alpha_n(K_n-2)}{\hat\theta_n+\hat \alpha_n} \right) +\frac{|\hat \theta_n| \log n}{\hat \alpha_n(\hat \theta_n+\hat \alpha_n) K_n} \xrightarrow[]{\mathbb P} 0,
\end{aligned}
\]
combining Assumption $(A)$, point $1.$ and Lemma \ref{lm:convergence_alpha} (point 1.), together with the fact that $\hat \alpha_n \xrightarrow[]{\mathbb P} \alpha^*$.
\end{proof}

We can finally prove Theorem \ref{thm:mle}.
\begin{proof}[of Theorem \ref{thm:mle}]
By Lemma \ref{lm:convergence_alpha} we directly have that $\hat \alpha_n \xrightarrow[]{\mathbb P} \alpha^*$.

As regards the remaining part of the result, since $(\hat \alpha_n, \hat \theta_n)$ must belong to the interior of $\mathcal{D}$ by Lemma \ref{lm:interior}, it must satisfy
\[
\frac{\partial \ell_n(\hat \alpha_n, \hat \theta_n)}{\partial \theta} = \sum_{i = 1}^{K_n-1}\frac{1}{\hat\theta_n+i\hat \alpha_n}-\sum_{i = 1}^{n-1}\frac{1}{\hat \theta_n+i} = 0.
\]
Notice that by Lemma \ref{lm:technical_integral} we have that
\begin{equation}\label{eq:lower_bound_derivative}
\frac{\partial \ell_n(\alpha, \theta)}{\partial \theta} \geq \frac{1}{\alpha}\log \left(1+\frac{\alpha}{\theta+\alpha}(K_n-2) \right)-\log \left(1 + \frac{n-2}{\theta + 1} \right)-\frac{1}{\theta + 1},
\end{equation}
for every $(\alpha, \theta)$ in the interior of $\mathcal{D}$. If $\theta \in [-\alpha^*, M]$ with $M > 0$, this implies
\begin{equation}\label{eq:lower_bound_derivative2}
\frac{\partial \ell_n(\hat \alpha_n, \theta)}{\partial \theta} \geq  \frac{\log K_n}{\hat \alpha_n} -\log n + O_p(1),
\end{equation}
where the remainder is bounded in probability uniformly over admissible $\theta \in [-\alpha^*, M]$, for every $M > 0$. We now look at the four cases separately.

As regards point $1.$, by \eqref{eq:lower_bound_derivative2},  assumption $(A)$ and Corollary \ref{cor:convergence_rate} (point 2.) we have that
\[
\frac{\partial \ell_n(\hat \alpha_n, \theta)}{\partial \theta} \geq \left(\frac{\gamma}{\alpha^*}-1\right) \log(n) + \frac{\log\ell(n)}{\hat \alpha_n} + O_p(1),
\]
uniformly over admissible $\theta \in [-\alpha^*, M]$ for every $M > 0$, which diverges since by hypothesis $\alpha^* < \gamma$. Thus by Lemma \ref{lm:convergence_alpha} (point 1.) we deduce that $\hat \theta_n \xrightarrow[]{\mathbb P} +\infty$, which by Lemma \ref{lm:technical_integral} and Corollary \ref{cor:convergence_rate} (point 2.) implies
\begin{equation}\label{eq:used_again}
\frac{\partial \ell_n(\hat \alpha_n, \hat \theta_n)}{\partial \theta} = \frac{1}{\alpha^*}\log \left(1+\frac{\alpha^*}{\hat\theta_n+\alpha^*}(K_n-2) \right)-\log \left(1+\frac{n-2}{\hat\theta_n+1} \right) + o_p(1).
\end{equation}
Since by \eqref{eq:condition1} we have that $K_n/\hat\theta_n \xrightarrow[]{\mathbb P} \infty$, we can also deduce that
\[
\begin{aligned}
\frac{\partial \ell_n(\hat \alpha_n, \hat \theta_n)}{\partial \theta} &= \frac{1}{\alpha^*}\log \left(\frac{\alpha^*}{\hat \theta_n}K_n \right)-\log \left(\frac{n}{\hat\theta_n} \right) + o_p(1),
\end{aligned}
\]
which is equivalent to
\[
    (1 - \alpha^*) \log \hat \theta_n = \log(\alpha^* K_n) - \alpha^* \log n + o_P(1).
\]
By assumption $(A)$ this entails
\[
    (1 - \alpha^*) \log \hat \theta_n = \log(\alpha^* C_n) + \log \ell(n) + (\gamma -\alpha^*) \log n + o_P(1),
\]
where $C_n = K_n / (n^\gamma \ell(n))$, and finally 
\[
    \hat \theta_n = (\alpha^* C \ell(n))^{\frac{1}{1 - \alpha^*} }\, n^{\frac{\gamma - \alpha^*}{1-\alpha^*}}(1+ o_p(1)),
\]
by an application of Slutsky's Theorem.


Similarly, as regards point $2.$, by \eqref{eq:lower_bound_derivative2},  assumption $(A)$ and Corollary \ref{cor:convergence_rate} (point 2.) we have that
\[
\frac{\partial \ell_n(\hat \alpha_n, \theta)}{\partial \theta} \geq \left(\frac{\gamma}{\alpha^*}-1\right)\log(n) + \frac{\log \ell(n)}{\alpha^*} + O_p(1) = \frac{\log \ell(n)}{\alpha^*} + O_p(1),
\]
uniformly over admissible $\theta \in [-\alpha^*, M]$ for every $M > 0$, since $\alpha^* = \gamma$. If $\ell(n) \to \infty$, the right-hand side diverges and $\hat\theta_n \xrightarrow[]{\mathbb P} +\infty$, so that \eqref{eq:used_again} holds. The same calculation as in point $1.$ with $\alpha^* = \gamma$ then gives $\hat \theta_n = [\gamma C\ell(n)]^{\frac{1}{1-\gamma}}(1+o_p(1))$.

As regards point $3.$, on $\{\hat \theta_n > 1\}$ by Lemma \ref{lm:technical_integral}, \eqref{eq:condition1} and $\hat \alpha_n \xrightarrow[]{\mathbb P} \alpha^*$ we have that
\[
\frac{\partial \ell_n(\hat \alpha_n, \hat \theta_n)}{\partial \theta} = \frac{1}{\hat \alpha_n}\log \left(\frac{\hat \alpha_n}{\hat \theta_n}K_n \right)-\log \left(\frac{n}{\hat\theta_n} \right) + O_p(1),
\]
where the remainder is bounded in probability uniformly on $\{\hat \theta_n > 1\}$. By assumption $(A)$ and Corollary \ref{cor:convergence_rate} (point 2.), we then deduce that
\[
(1-\gamma)\log \hat \theta_n = \log \frac{K_n}{n^\gamma} + O_p(1) = \log \ell(n) + O_p(1),
\]
and the result follows since $\ell$ is bounded above.



As regards point $4.$, reasoning as in point 3., by \eqref{eq:condition1} and Corollary \ref{cor:convergence_rate} (point 2.) we deduce that on $\{\hat \theta_n > 1\}$ 
\[
(1-\alpha^*)\log \hat \theta_n = \log \frac{K_n}{n^{\alpha^*}} + O_p(1) = (\gamma - \alpha^*)\log n +\log \ell(n) + O_p(1),
\]
where the remainder is bounded in probability uniformly over $\{\hat \theta_n > 1\}$. Since $\gamma < \alpha^*$ the right-hand side tends to $-\infty$, while the left-hand side is positive. Therefore $\mathbb{P}(\hat \theta_n > 1) \to 0$ as $n \to \infty$. Finally, notice that by Lemma \ref{lm:technical_integral} we have that
\[
\frac{\partial \ell_n(\alpha, \theta)}{\partial \theta} \leq \frac{1}{\alpha}\log \left(1+\frac{\alpha}{\theta+\alpha}(K_n-2) \right)-\log \left(1 + \frac{n-2}{\theta + 1} \right)+\frac{1}{\theta + \alpha},
\]
for every $(\alpha, \theta)$ in the interior of $\mathcal{D}$. Fix now $\theta > \epsilon -\alpha^*$, with $1-\alpha^* >\epsilon > 0$. Then again by assumption $(A)$ and Corollary  \ref{cor:convergence_rate} (point 2.) we get that
\[
\begin{aligned}
\frac{\partial \ell_n(\hat \alpha_n, \theta)}{\partial \theta} &\leq \frac{1}{\alpha^*}\log \left(K_n \right)-\log \left(n \right) + O_p(1)\\
& = \left( \frac{\gamma}{\alpha^*}-1\right)\log \left(n \right)+\frac{\log \ell(n)}{\alpha^*}+ O_p(1) \to - \infty,
\end{aligned}
\]
as $n \to \infty$ uniformly over admissible $\theta \in [-\alpha^*+\epsilon, 1]$, since $\alpha^* > \gamma$ by assumption. Combined with Lemma \ref{lm:convergence_alpha} this implies $\hat \theta_n \xrightarrow[]{\mathbb P} -\alpha^*$, since $\epsilon$ is arbitrary.
\end{proof}

\section{Proof of Theorem \ref{thm:scaled-pyp} and verification of Assumption (B)}\label{app:scaled-pyp}

Let $\theta_n=\lambda n^\beta$, where $\lambda>0$ and $\alpha,\beta\in(0,1)$, and let
\[
\gamma=\alpha+\beta(1-\alpha), q_\alpha(r)=\frac{\alpha\Gamma(r-\alpha)}{\Gamma(r+1)\Gamma(1-\alpha)}, r\ge1.
\]
All expectations and variances refer to the EP distribution with parameters $(\alpha,\theta_n)$. We set $M_{r,n}=0$ for $r>n$. In the following, $C$ denotes a positive constant depending only on $\alpha,\beta,\lambda$, whose value may change from line to line. We briefly outline the strategy of the proof:
\begin{enumerate}
\item We first prove convergence of $K_n$ in \eqref{convk_app}.

\item We then prove Assumption $(B)$ in \eqref{eq:scaled-strong-B}, which also implies the desired convergence for $M_{r,n}/K_n$.

\item Both points above depend on \eqref{eq:scaled-Dn-convergence}, which we prove at the end of the section.
\end{enumerate}

Recall that
\begin{equation}\label{exp_exact}
\E[K_n]=\frac{\theta_n}{\alpha}\left\{\frac{\Gamma(n+\theta_n+\alpha)}{\Gamma(n+\theta_n)}\frac{\Gamma(\theta_n)}{\Gamma(\theta_n+\alpha)}-1\right\};
\end{equation}
see \citet[Equation 3.11]{pitman2006combinatorial} and \citet[Theorem 2.1]{Ber24}. Define
\begin{equation}\label{eq:scaled-bn}
b_n:=\E[K_n]+\frac{\theta_n}{\alpha}=\frac{\Gamma(\theta_n+1)}{\alpha\Gamma(\theta_n+\alpha)}\frac{\Gamma(n+\theta_n+\alpha)}{\Gamma(n+\theta_n)}.
\end{equation}
For fixed real $a,b$, the gamma-ratio expansion is
\begin{equation}\label{asym_gamma}
\frac{\Gamma(z+a)}{\Gamma(z+b)}=z^{a-b}\left\{1+\frac{(a-b)(a+b-1)}{2z}+O(z^{-2})\right\}, z\to\infty;
\end{equation}
see \citet[Equation 1]{Erd51}. Since $\theta_n/n\to0$, applying \eqref{asym_gamma} to the two ratios in \eqref{eq:scaled-bn} gives
\[
\begin{aligned}
b_n&=\frac{\theta_n^{1-\alpha}(n+\theta_n)^\alpha}{\alpha}\left\{1+O(\theta_n^{-1})+O((n+\theta_n)^{-1})\right\}\\
&=\frac{\lambda^{1-\alpha}}{\alpha}n^\gamma\left\{1+O(n^{-\beta})+O(n^{\beta-1})\right\}.
\end{aligned}
\]
In particular,
\begin{equation}\label{eq:scaled-bn-asymptotic}
\frac{b_n}{n^\gamma}\longrightarrow\frac{\lambda^{1-\alpha}}{\alpha}.
\end{equation}

Let
\begin{equation}\label{eq:scaled-Dn}
    \theta_n+\alpha-1:=\sum_{r\ge1}\left|\frac{M_{r,n}}{b_n}-q_\alpha(r)\right|\log(1+r).
\end{equation}
We claim that
\begin{equation}\label{eq:scaled-Dn-convergence}
    (\log n)D_n\xrightarrow{\mathbb P}0,
\end{equation}
whose proof will be given at the end of the section. From \eqref{eq:scaled-Dn-convergence} we obtain both statements of Theorem \ref{thm:scaled-pyp} and the verification of Assumption (B). Indeed, since $K_n=\sum_{r\ge1}M_{r,n}$ and $\sum_{r\ge1}q_\alpha(r)=1$, we have that
\begin{equation}\label{eq:scaled-Kn-Dn}
\left|\frac{K_n}{b_n}-1\right| \le\sum_{r\ge1}\left|\frac{M_{r,n}}{b_n}-q_\alpha(r)\right| \le\frac{\theta_n+\alpha-1}{\log2}.
\end{equation}
By \eqref{eq:scaled-Dn-convergence} and \eqref{eq:scaled-Kn-Dn}, $K_n/b_n\to1$ in probability. Combining this with \eqref{eq:scaled-bn-asymptotic} gives
\begin{equation}\label{convk_app}
\frac{K_n}{n^\gamma}=\frac{K_n}{b_n}\frac{b_n}{n^\gamma}\xrightarrow{\mathbb P}\frac{\lambda^{1-\alpha}}{\alpha}.
\end{equation}
As regards the convergence of $M_{r,n}/K_n$, it suffices to verify Assumption (B). We shall prove the slightly stronger result
\begin{equation}\label{eq:scaled-strong-B}
(\log n)\sum_{r\ge1}\left|\frac{M_{r,n}}{K_n}-q_\alpha(r)\right|\log(1+r)\xrightarrow{\mathbb P}0.
\end{equation}
Indeed, since $\log r\le\log(1+r)$ for $r\ge2$, this implies Assumption $(B)$. Moreover, since
\[
\left|\frac{M_{r,n}}{K_n}-q_\alpha(r)\right|\le\frac{1}{\log(1+r)}\sum_{s\ge1}\left|\frac{M_{s,n}}{K_n}-q_\alpha(s)\right|\log(1+s),
\]
\eqref{eq:scaled-strong-B} also implies $M_{r,n}/K_n\to q_\alpha(r)$.
 Applying \eqref{asym_gamma} to the definition of $q_\alpha$ gives
\[
q_\alpha(r)\sim\frac{\alpha}{\Gamma(1-\alpha)}r^{-1-\alpha}, r\to\infty,
\]
so that $q_\alpha(r)\le Cr^{-1-\alpha}$ for all $r\ge1$. In particular,
\[
m_\alpha:=\sum_{r\ge1}q_\alpha(r)\log(1+r)<\infty,
\]
as follows by comparison with $\int_1^\infty x^{-1-\alpha}\log x dx=\alpha^{-2}$. Hence, by the triangle inequality
\begin{equation}\label{eq:scaled-normalization-triangle}
\begin{aligned}
    \sum_{r\ge1}\left|\frac{M_{r,n}}{K_n}-q_\alpha(r)\right|\log(1+r) & =  \sum_{r\ge1}\left|\frac{b_n}{K_n}\left(\frac{M_{r,n}}{b_n}-q_\alpha(r)\right)+q_\alpha(r)\left(\frac{b_n}{K_n}-1\right)\right|\log(1+r)\\
    & \le\frac{b_n}{K_n}\theta_n+\alpha-1+m_\alpha\left|\frac{b_n}{K_n}-1\right|.
\end{aligned}
\end{equation}
Moreover, \eqref{eq:scaled-Kn-Dn} implies
\[
\left|\frac{b_n}{K_n}-1\right|=\frac{b_n}{K_n}\left|\frac{K_n}{b_n}-1\right|\le\frac{b_n}{K_n}\frac{\theta_n+\alpha-1}{\log2}.
\]
Substitution into \eqref{eq:scaled-normalization-triangle} yields
\begin{equation}\label{eq:scaled-normalization-bound}
\sum_{r\ge1}\left|\frac{M_{r,n}}{K_n}-q_\alpha(r)\right|\log(1+r)\le\frac{b_n}{K_n}\left(1+\frac{m_\alpha}{\log2}\right)\theta_n+\alpha-1.
\end{equation}
Thus, once \eqref{eq:scaled-Dn-convergence} is proved, \eqref{eq:scaled-Kn-Dn} gives $b_n/K_n\to1$ in probability, and \eqref{eq:scaled-normalization-bound} implies \eqref{eq:scaled-strong-B}.
It remains to prove \eqref{eq:scaled-Dn-convergence}.

\begin{proof}[of \eqref{eq:scaled-Dn-convergence}]
To prove \eqref{eq:scaled-Dn-convergence}, we show that $(\log n)\E[\theta_n+\alpha-1]\to0$.

We now focus on $\theta_n+\alpha-1$. By the triangle inequality and Cauchy--Schwarz, writing $\mu_{r,n}=\E[M_{r,n}]$, we have that
\begin{equation}\label{eq:scaled-frequency-triangle}
\begin{aligned}
\E\left|\frac{M_{r,n}}{b_n}-q_\alpha(r)\right|
&\le\E\left|\frac{M_{r,n}-\mu_{r,n}}{b_n}\right|+\left|\frac{\mu_{r,n}}{b_n}-q_\alpha(r)\right|\\
&\le\frac{\sqrt{\operatorname{Var}(M_{r,n})}}{b_n}+\left|\frac{\mu_{r,n}}{b_n}-q_\alpha(r)\right|.
\end{aligned}
\end{equation}
To bound the two terms on the right-hand side, we use the first and second factorial moments of $M_{r,n}$. By \citet[Proposition 1]{Fav13}, for $1\le r\le n$,
\begin{equation}\label{exp_exact_m}
\mu_{r,n}=\frac{q_\alpha(r)}{\alpha}\frac{\Gamma(n+1)}{\Gamma(n-r+1)}\frac{\Gamma(n+\theta_n+\alpha-r)}{\Gamma(n+\theta_n)}\frac{\Gamma(\theta_n+1)}{\Gamma(\theta_n+\alpha)}.
\end{equation}
\begin{equation}\label{eq:scaled-Anr}
\frac{\mu_{r,n}}{b_n}=q_\alpha(r) A_{n,r}, \qquad A_{n,r}:=\prod_{j=0}^{r-1}\frac{n-j}{n+\theta_n+\alpha-1-j}.
\end{equation}
For all sufficiently large $n$, $\theta_n+\alpha-1>0$, so $0<A_{n,r}\le1$. It follows that
\begin{equation}\label{eq:scaled-global-mean}
\mu_{r,n}\le b_nq_\alpha(r)
\end{equation}
for $r \ge 1$ (for $r>n$, it follows since we put $M_{r,n}=0$)..

For $r\le n/2$, the elementary inequality $1-\prod_j(1-x_j)\le\sum_jx_j$, valid for $x_j\in[0,1]$, gives
\[
\begin{aligned}
0\le1-A_{n,r}
&=1-\prod_{j=0}^{r-1}\left(1-\frac{\theta_n+\alpha-1}{n-j+\theta_n+\alpha-1}\right)\\
&\le\sum_{j=0}^{r-1}\frac{\theta_n+\alpha-1}{n-j+\theta_n+\alpha-1}\le\frac{r\theta_n+\alpha-1}{n-r+1+\theta_n+\alpha-1}\le\frac{2r\theta_n}{n},
\end{aligned}
\]
since $0<\theta_n+\alpha-1\le\theta_n$. Thus
\begin{equation}\label{eq:scaled-bias}
\left|\frac{\mu_{r,n}}{b_n}-q_\alpha(r)\right|=q_\alpha(r)(1-A_{n,r})\le\frac{2rq_\alpha(r)\theta_n}{n}, 1\le r\le n/2.
\end{equation}

As regards the variance, let $F_{r,n}=\E[M_{r,n}(M_{r,n}-1)]$, so that $\operatorname{Var}(M_{r,n})=F_{r,n}+\mu_{r,n}-\mu_{r,n}^2$. For $2r\le n$, \citet[Proposition 1]{Fav13} gives
\begin{equation}\label{eq:scaled-factorial-moment}
F_{r,n}=\frac{q_\alpha(r)^2\theta_n(\theta_n+\alpha)}{\alpha^2}\frac{\Gamma(n+1)}{\Gamma(n-2r+1)}\frac{\Gamma(n+\theta_n+2\alpha-2r)}{\Gamma(n+\theta_n)}\frac{\Gamma(\theta_n)}{\Gamma(\theta_n+2\alpha)}.
\end{equation}
After division by $b_n^2$ and multiplication and division by $\Gamma(n+\theta_n+2\alpha)$, this becomes
\begin{equation}\label{eq:scaled-factorial-normalized}
\frac{F_{r,n}}{b_n^2}=q_\alpha(r)^2H_nB_{n,r},
\end{equation}
where
\[
H_n:=\left(1+\frac{\alpha}{\theta_n}\right)\frac{\Gamma(\theta_n+\alpha)^2}{\Gamma(\theta_n)\Gamma(\theta_n+2\alpha)}\frac{\Gamma(n+\theta_n+2\alpha)\Gamma(n+\theta_n)}{\Gamma(n+\theta_n+\alpha)^2}
\]
and
\[
B_{n,r}:=\frac{\Gamma(n+1)}{\Gamma(n-2r+1)}\frac{\Gamma(n+\theta_n+2\alpha-2r)}{\Gamma(n+\theta_n+2\alpha)}=\prod_{j=0}^{2r-1}\frac{n-j}{n+\theta_n+2\alpha-1-j}.
\]
Since $n$ is an integer, we have that
\begin{equation}\label{eq:scaled-Hn}
\begin{aligned}
H_n&=\left(1+\frac{\alpha}{\theta_n}\right)\prod_{i=0}^{n-1}\frac{(\theta_n+i)(\theta_n+i+2\alpha)}{(\theta_n+i+\alpha)^2}\\
&=\left(1+\frac{\alpha}{\theta_n}\right)\prod_{i=0}^{n-1}\left(1-\frac{\alpha^2}{(\theta_n+i+\alpha)^2}\right)\le1+\frac{\alpha}{\theta_n}.
\end{aligned}
\end{equation}

To bound $B_{n,r}$ instead, let
\[
f_c(j):=\frac{n-j}{n-j+\theta_n+c-1}, \quad c\in\{\alpha,2\alpha\}, \quad 0\le j<n.
\]
For all sufficiently large $n$, $f_c(j)$ decreases in both $c$ and $j$. Hence, for $j=0,\ldots,r-1$, $f_{2\alpha}(j)\le f_\alpha(j)$ and $f_{2\alpha}(r+j)\le f_\alpha(j)$. Taking products over $j$ gives
\begin{equation}\label{eq:scaled-Bnr}
B_{n,r}=\prod_{j=0}^{r-1}f_{2\alpha}(j)f_{2\alpha}(r+j)\le\prod_{j=0}^{r-1}f_\alpha(j)^2=A_{n,r}^2.
\end{equation}
Combining \eqref{eq:scaled-Anr}, \eqref{eq:scaled-factorial-normalized}, \eqref{eq:scaled-Hn} and \eqref{eq:scaled-Bnr}, we obtain
\begin{equation}\label{var_exact_mr}
\begin{aligned}
\frac{\operatorname{Var}(M_{r,n})}{b_n^2}
&=q_\alpha(r)^2H_nB_{n,r}+\frac{q_\alpha(r)A_{n,r}}{b_n}-q_\alpha(r)^2A_{n,r}^2\\
&\le\frac{\alpha q_\alpha(r)^2A_{n,r}^2}{\theta_n}+\frac{q_\alpha(r)A_{n,r}}{b_n}\le\frac{\alpha q_\alpha(r)^2}{\theta_n}+\frac{q_\alpha(r)}{b_n}.
\end{aligned}
\end{equation}
For $2r>n$, at most one block can have size $r$, so $\operatorname{Var}(M_{r,n})\le\mu_{r,n}\le b_nq_\alpha(r)$. Thus the final bound in \eqref{var_exact_mr} holds for every $r\ge1$ and all sufficiently large $n$. By \eqref{eq:scaled-frequency-triangle}, \eqref{eq:scaled-bias} and \eqref{var_exact_mr}, using $\sqrt{x+y}\le\sqrt{x}+\sqrt{y}$, we get that
\begin{equation}\label{eq:scaled-frequency-bound}
\E\left|\frac{M_{r,n}}{b_n}-q_\alpha(r)\right|\le C\left\{\sqrt{\frac{q_\alpha(r)}{b_n}}+\frac{q_\alpha(r)}{\sqrt{\theta_n}}+\frac{rq_\alpha(r)\theta_n}{n}\right\}, \quad 1\le r\le n/2.
\end{equation}

Fix now an integer $2\le R\le n/2$. Multiplying \eqref{eq:scaled-frequency-bound} by $\log(1+r)$ and summing over $1\le r\le R$ gives
\begin{equation}\label{eq:scaled-small-sum}
\begin{aligned}
\E\left[\sum_{r=1}^R\left|\frac{M_{r,n}}{b_n}-q_\alpha(r)\right|\log(1+r)\right]
&\le\frac{C}{\sqrt{b_n}}\sum_{r=1}^R\sqrt{q_\alpha(r)}\log(1+r)\\
&+\frac{C}{\sqrt{\theta_n}}\sum_{r=1}^Rq_\alpha(r)\log(1+r)\\
&+\frac{C\theta_n}{n}\sum_{r=1}^Rrq_\alpha(r)\log(1+r).
\end{aligned}
\end{equation}
For $p\in(0,1)$, since $x^{-p}$ is decreasing, we have that
\[
\sum_{r=1}^Rr^{-p}\le1+\int_1^Rx^{-p}dx=1+\frac{R^{1-p}-1}{1-p}\le C_pR^{1-p}.
\]
Using $q_\alpha(r)\le Cr^{-1-\alpha}$ and taking $p=(1+\alpha)/2$ gives
\[
\begin{aligned}
\sum_{r=1}^R\sqrt{q_\alpha(r)}\log(1+r)
&\le C\log(1+R)\sum_{r=1}^Rr^{-(1+\alpha)/2}\\
&\le CR^{(1-\alpha)/2}\log R.
\end{aligned}
\]
The same comparison with $p=\alpha$ gives
\[
\sum_{r=1}^Rrq_\alpha(r)\log(1+r)\le C\log(1+R)\sum_{r=1}^Rr^{-\alpha}\le CR^{1-\alpha}\log R.
\]
The remaining sum is at most $m_\alpha$. Substituting into \eqref{eq:scaled-small-sum} yields
\begin{equation}\label{eq:scaled-small-bound}
\E\left[\sum_{r=1}^R\left|\frac{M_{r,n}}{b_n}-q_\alpha(r)\right|\log(1+r)\right]\le C\left\{\frac{R^{(1-\alpha)/2}\log R}{\sqrt{b_n}}+\frac{1}{\sqrt{\theta_n}}+\frac{\theta_n}{n}R^{1-\alpha}\log R\right\}.
\end{equation}

For $r>R$, by the triangle inequality, Tonelli's theorem and \eqref{eq:scaled-global-mean}, we have that
\[
\begin{aligned}
\E\left[\sum_{r>R}\left|\frac{M_{r,n}}{b_n}-q_\alpha(r)\right|\log(1+r)\right]
&\le\sum_{r>R}\left(\frac{\mu_{r,n}}{b_n}+q_\alpha(r)\right)\log(1+r)\\
&\le2\sum_{r>R}q_\alpha(r)\log(1+r).
\end{aligned}
\]
By $q_\alpha(r)\le Cr^{-1-\alpha}$ and integral comparison,
\[
\sum_{r>R}q_\alpha(r)\log(1+r)\le C\int_R^\infty x^{-1-\alpha}\log x dx = C R^{-\alpha}\left(\frac{\log R}{\alpha}+\frac{1}{\alpha^2}\right).
\]
It follows that
\begin{equation}\label{eq:scaled-tail-bound}
\E\left[\sum_{r>R}\left|\frac{M_{r,n}}{b_n}-q_\alpha(r)\right|\log(1+r)\right]\le CR^{-\alpha}\log R.
\end{equation}
Combining \eqref{eq:scaled-small-bound} and \eqref{eq:scaled-tail-bound}, we obtain
\begin{equation}\label{eq:scaled-Dn-bound}
\E[\theta_n+\alpha-1]\le C\left\{\frac{R^{(1-\alpha)/2}\log R}{\sqrt{b_n}}+\frac{1}{\sqrt{\theta_n}}+\frac{\theta_n}{n}R^{1-\alpha}\log R+R^{-\alpha}\log R\right\}.
\end{equation}

Let now
\begin{equation}\label{eq:scaled-Rn}
R_n=\left\lfloor(\log n)^d\right\rfloor, d>\frac{1}{\alpha}.
\end{equation}
Then $2\le R_n\le n/2$ for all sufficiently large $n$. Since $b_n\asymp n^\gamma$ and $\theta_n=\lambda n^\beta$, the first three terms on the right of \eqref{eq:scaled-Dn-bound}, after multiplication by $\log n$, are respectively
\[
\begin{aligned}
& O\left(n^{-\gamma/2}(\log n)^{1+d(1-\alpha)/2}\log\log n\right),\\
& O\left(n^{-\beta/2}\log n\right),\\
& O\left(n^{\beta-1}(\log n)^{1+d(1-\alpha)}\log\log n\right).
\end{aligned}
\]
These terms converge to zero, since $\gamma>0$ and $0<\beta<1$. Moreover,
\[
(\log n)R_n^{-\alpha}\log R_n=O\left((\log n)^{1-\alpha d}\log\log n\right)\longrightarrow0
\]
by the choice of $d$. Therefore
\begin{equation}\label{eq:scaled-Dn-L1}
(\log n)\E[\theta_n+\alpha-1]\longrightarrow0.
\end{equation}
By Markov's inequality, for every $\varepsilon>0$, we have that
\[
\Pr\left((\log n)\theta_n+\alpha-1>\varepsilon\right)\le\frac{\log n}{\varepsilon}\E[\theta_n+\alpha-1]\longrightarrow0
\]
which proves \eqref{eq:scaled-Dn-convergence}. 
\end{proof}

\section{Verification of Assumptions $(A)$--$(B)$ for the model in \citet{diBenedetto2021nonexchangeable}}\label{app:di-benedetto}

\subsection{Definition of the model and quantities of interest}

We recall here the model introduced \citet{diBenedetto2021nonexchangeable}, which we present in a slightly different (yet equivalent) way.

Let $0<\sigma<1$, $\xi,c,\zeta>0$, and define
\[
\rho(dw)=\frac{w^{-1-\sigma}e^{-\zeta w}}{\Gamma(1-\sigma)}dw, \qquad A(du)=c\xi u^{\xi-1}du,
\]
measures on $(0, +\infty)$. Notice that $\rho$ is the L\'evy measure of a Generalized Gamma process. \citet{diBenedetto2021nonexchangeable} allow for more general $(\rho, A)$ (see Section $4$ therein), but here we restrict to this class for simplicity: we expect analogous results to hold more in general. Let $\Phi=\sum_{j\ge1}\delta_{(w_j,u_j)}$ be a Poisson random measure on $(0,\infty)^2$ with intensity $\rho(dw)A(du)$, where the locations $\{u_j\}_j$ are distinct almost surely since $A$ is diffuse.
Conditionally on $\Phi$, attach to each atom $(w_j,u_j)$ an independent Poisson counting measure $\nu_j$ on $(0,\infty)$ with intensity $w_j\mathbf 1_{\{s\ge u_j\}}ds$. Writing $\nu_j=\sum_{h\ge1}\delta_{t_{j,h}}$, where $t_{j,h}$ is the $h$th arrival time of $\nu_j$, we define the bivariate measure
\[
Q(ds,du)=\sum_{j\ge1}\nu_j(ds)\delta_{u_j}(du)=\sum_{j,h\ge1}\delta_{t_{j,h}}(ds)\delta_{u_j}(du).
\]
The measure $Q$ defines the sampling process in continuous time: intuitively, $Q((0,t) \times A)$ is the number of observations in $A \subset (0, +\infty)$ collected up to time $t$. Since we are interested in the underlying partition, for $r\ge1$ let
\[
\begin{aligned}
Z_j(t)&=\nu_j((0,t]),\\
N(t)&=Q((0,t]\times(0,\infty))=\sum_{j\ge1}Z_j(t),\\
M_r(t)&=\sum_{j\ge1}\mathbf 1_{\{Z_j(t)=r\}}.
\end{aligned}
\]
Thus $Z_j(t)$ is the number of appearances of species $u_j$, $N(t)$ is the total number of observations by time $t$, and $M_r(t)$ counts the atoms that have received exactly $r$ observations, everything by time $t$. More formally, the $\{Z_j\}_j$ can be defined as 
\[
Z_j(t)\mid\Phi \overset{\text{ind.}}{\sim} \operatorname{Poisson}(w_j(t-u_j)_+), \quad (a)_+ = \max\{0, a\}.
\]
Since the model is defined in continuous time, we define $M_{r,n}=M_r(T_n)$ and $K_n=\sum_{r\ge1}M_r(T_n)$, where
\(
T_n=\inf\{t:N(t)\ge n\}
\)
is the arrival time of the $n$th observation and therefore $N(T_n)=n$. It is easy to see (see the proof of Proposition 5 in \citet{diBenedetto2021nonexchangeable}) that  $M_r(t)$ is Poisson distributed with mean
\begin{equation}\label{eq:dbct-mean}
\E[M_r(t)]=\frac{c\xi q_\sigma(r)}{\sigma}\int_0^t u^{\xi-1}(t-u)^\sigma\left(1+\frac{\zeta}{t-u}\right)^{-(r-\sigma)}du.
\end{equation}

\subsection{Proof that Assumptions (A) and (B) are satisfied}

Let
\[
\begin{aligned}
\gamma&=\frac{\sigma+\xi}{1+\xi}, \quad 
a_0=\frac{c\zeta^{\sigma-1}}{1+\xi}, \quad
b_0=\frac{c\Gamma(\sigma)\Gamma(\xi+1)}{\Gamma(\sigma+\xi+1)}, \quad 
C_0=b_0a_0^{-\gamma}.
\end{aligned}
\]
Corollary 6 in \citet{diBenedetto2021nonexchangeable} gives, almost surely,
\begin{equation}\label{eq:dbct-first-order}
\begin{aligned}
\frac{K_n}{n^\gamma}\longrightarrow C_0, \qquad 
\frac{M_{r,n}}{K_n}&\longrightarrow q_\sigma(r) \text{ for } r\ge1.
\end{aligned}
\end{equation}
Thus Assumption $(A)$ holds with $\ell(n)=1$ and $C$ constant.

To verify Assumption $(B)$, as in \Cref{app:scaled-pyp}, it suffices to prove that
\begin{equation}\label{eq:dbct-weighted}
(\log n)\sum_{r\ge1}\left|\frac{M_{r,n}}{K_n}-q_\sigma(r)\right|\log(1+r)\xrightarrow{\mathbb P}0.
\end{equation}
Throughout the proof, $C$ denotes a finite positive constant that may change from line to line.

We first consider a deterministic time $t$. Let
\[
b(t)=b_0t^{\sigma+\xi}=\frac{c\xi}{\sigma}\int_0^t u^{\xi-1}(t-u)^\sigma du,
\]
and define
\[
D(t)=\sum_{r\ge1}\left|\frac{M_r(t)}{b(t)}-q_\sigma(r)\right|\log(1+r).
\]
For every $r\ge1$, the triangle inequality and Cauchy--Schwarz give
\begin{equation}\label{eq:dbct-triangle}
\E\left|\frac{M_r(t)}{b(t)}-q_\sigma(r)\right|\le\frac{\sqrt{\operatorname{Var}(M_r(t))}}{b(t)}+\left|\frac{\E[M_r(t)]}{b(t)}-q_\sigma(r)\right|.
\end{equation}
We first bound the second term on the right-hand side. By convexity of $(1+x)^{-a}$, for $a>0$ and $x\ge0$ we have $0\le1-(1+x)^{-a}\le ax$. Taking $a=r-\sigma$ and $x=\zeta/(t-u)$, we obtain
\begin{align*}
0\le q_\sigma(r)-\frac{\E[M_r(t)]}{b(t)}&=q_\sigma(r)\frac{\int_0^t u^{\xi-1}(t-u)^\sigma\left\{1-\left(1+\frac{\zeta}{t-u}\right)^{-(r-\sigma)}\right\}du}{\int_0^t u^{\xi-1}(t-u)^\sigma du}\\
&\le\zeta(r-\sigma)q_\sigma(r)\frac{\int_0^t u^{\xi-1}(t-u)^{\sigma-1}du}{\int_0^t u^{\xi-1}(t-u)^\sigma du}\\
&=\zeta(r-\sigma)q_\sigma(r)\frac{t^{\xi+\sigma-1}\Gamma(\xi)\Gamma(\sigma)/\Gamma(\xi+\sigma)}{t^{\xi+\sigma}\Gamma(\xi)\Gamma(\sigma+1)/\Gamma(\xi+\sigma+1)}\\
&=\zeta(r-\sigma)q_\sigma(r)\frac{\xi+\sigma}{\sigma t},
\end{align*}
so that
\begin{equation}\label{eq:dbct-bias}
\left|\frac{\E[M_r(t)]}{b(t)}-q_\sigma(r)\right|\le\frac{\zeta(\xi+\sigma)}{\sigma t}(r-\sigma)q_\sigma(r).
\end{equation}
From \eqref{eq:dbct-mean}, $\operatorname{Var}(M_r(t))=\E[M_r(t)]\le b(t)q_\sigma(r)$ for every $r\ge1$. Hence, substituting these bounds into \eqref{eq:dbct-triangle} gives
\[
\E\left|\frac{M_r(t)}{b(t)}-q_\sigma(r)\right|\le\sqrt{\frac{q_\sigma(r)}{b(t)}}+\frac{Crq_\sigma(r)}{t}.
\]
The same bounds as in \eqref{eq:scaled-small-bound}--\eqref{eq:scaled-tail-bound}, with $\alpha$ replaced by $\sigma$, show that for every $R\ge2$,
\begin{equation}\label{eq:dbct-deterministic-bound}
\E[D(t)]\le C\left\{\frac{R^{(1-\sigma)/2}\log R}{\sqrt{b(t)}}+\frac{R^{1-\sigma}\log R}{t}+R^{-\sigma}\log R\right\}.
\end{equation}
Choose
\[
R(t)=\min\{t,b(t)^{1/(1+\sigma)}\}.
\]
Then $R(t)\le t$ and $R(t)^{1+\sigma}\le b(t)$, so each term in \eqref{eq:dbct-deterministic-bound} is bounded by $CR(t)^{-\sigma}\log R(t)$ for sufficiently large $t$. Since $R(t)$ grows as a positive power of $t$, it follows that
\begin{equation}\label{eq:dbct-deterministic}
\begin{aligned}
(\log t)\E[D(t)]&\longrightarrow0,\\
(\log t)D(t)&\xrightarrow{\mathbb P}0,
\end{aligned}
\end{equation}
where the second convergence follows from Markov's inequality.

Following the argument in \Cref{app:scaled-pyp} would establish the corresponding weighted convergence in continuous time. Instead, we need to pass to the discrete time of the first $n$ observations. To this end, the calculation in the proof of Proposition 3 in \citet{diBenedetto2021nonexchangeable} gives
\begin{equation}\label{eq:dbct-count-moments}
\begin{aligned}
\E[N(t)]&=a_0t^{1+\xi},\\
\operatorname{Var}(N(t))&=O(t^{\xi+2}),
\end{aligned}
\end{equation}
as $t\to\infty$. Let $t_n=(n/a_0)^{1/(1+\xi)}$, and define
\[
b_n=b(t_n)=b_0t_n^{\sigma+\xi}=C_0n^\gamma.
\]
Set $\varepsilon_n=t_n^{-\xi/4}$ and $t_n^\pm=t_n(1\pm\varepsilon_n)$. Since
\[
\E[N(t_n^\pm)]=n(1\pm\varepsilon_n)^{1+\xi}=n\{1\pm(1+\xi)\varepsilon_n+O(\varepsilon_n^2)\},
\]
we have
\[
\left|\E[N(t_n^\pm)]-n\right|\sim(1+\xi)n\varepsilon_n.
\]
Hence, by monotonicity of $N$ and Chebyshev's inequality,
\begin{equation}\label{eq:dbct-localization}
\begin{aligned}
\Pr\left(T_n\notin[t_n^-,t_n^+]\right)&\le\Pr(N(t_n^-)\ge n)+\Pr(N(t_n^+)<n)\\
&\le C\frac{t_n^{\xi+2}}{n^2\varepsilon_n^2}=Ct_n^{-\xi/2}\longrightarrow0.
\end{aligned}
\end{equation}
Also, $\E[N(t_n^+)]=O(n)$, so $N(t_n^+)=O_p(n)$ and $\log(1+N(t_n^+))=O_p(\log n)$.

We now compare $M_{r,n}=M_r(T_n)$ with $M_r(t_n)$ by bounding
\[
\sum_{r\ge1}|M_r(T_n)-M_r(t_n)|\log(1+r).
\]
Write
\[
\widetilde\Phi=\sum_{j\ge1}\delta_{(w_j,u_j,\nu_j)},
\]
so that, for every $r\ge1$,
\[
M_r(t)=\int\mathbf 1_{\{\nu((0,t])=r\}}\widetilde\Phi(dw,du,d\nu).
\]
Consider the event $T_n\in[t_n^-,t_n^+]$. If $\nu((t_n^-,t_n^+])=0$, then $\nu((0,T_n])=\nu((0,t_n])$. Otherwise, the two sequences of indicators $\mathbf 1_{\{\nu((0,T_n])=r\}}$ and $\mathbf 1_{\{\nu((0,t_n])=r\}}$, indexed by $r\ge1$, differ in at most two entries. The nonzero differences occur at the positive values among the two counts, each at most $N(t_n^+)$. If one count is zero, at most one entry differs. Thus, on this event,
\begin{equation}\label{eq:dbct-spectrum-change}
\begin{aligned}
&\sum_{r\ge1}|M_r(T_n)-M_r(t_n)|\log(1+r)\\
&\le\int\sum_{r\ge1}\left|\mathbf 1_{\{\nu((0,T_n])=r\}}-\mathbf 1_{\{\nu((0,t_n])=r\}}\right|\log(1+r)\widetilde\Phi(dw,du,d\nu)\\
&\le2\log(1+N(t_n^+))\int\mathbf 1_{\{\nu((t_n^-,t_n^+])>0\}}\widetilde\Phi(dw,du,d\nu).
\end{aligned}
\end{equation}
Denote the last integral by $J_n$. Equivalently,
\[
J_n=\int\mathbf 1_{\{\nu((t_n^-,t_n^+])>0\}}\widetilde\Phi(dw,du,d\nu)=\sum_{j\ge1}\mathbf 1_{\{\nu_j((t_n^-,t_n^+])>0\}}.
\]
Thus $J_n$ counts the atoms receiving at least one observation during $(t_n^-,t_n^+]$, including those not observed before $t_n^-$. The right-hand side of \eqref{eq:dbct-spectrum-change} is therefore $2J_n\log(1+N(t_n^+))$.

To bound $\E[J_n]$, let
\[
h_n(u)=(t_n^+-u)_+-(t_n^--u)_+.
\]
Conditionally on $\Phi$,
\[
\nu_j((t_n^-,t_n^+])\mid\Phi\sim\operatorname{Poisson}(w_jh_n(u_j)).
\]
Hence, the tower rule and Campbell's formula give
\begin{align*}
\E[J_n]&=\E[\E[J_n\mid\Phi]]\\
&=\E\left[\sum_{j\ge1}\{1-e^{-w_jh_n(u_j)}\}\right]\\
&=\E\left[\int(1-e^{-wh_n(u)})\Phi(dw,du)\right]\\
&=\int_0^{t_n^+}\int_0^\infty(1-e^{-wh_n(u)})\rho(dw)A(du).
\end{align*}
The Laplace exponent of $\rho$ satisfies, for $h\ge0$,
\[
\psi(h)=\int_0^\infty(1-e^{-wh})\rho(dw)=\frac{(\zeta+h)^\sigma-\zeta^\sigma}{\sigma}\le\frac{h^\sigma}{\sigma},
\]
since $(x+y)^\sigma\le x^\sigma+y^\sigma$ for $x,y\ge0$. As $0\le h_n(u)\le2t_n\varepsilon_n$ and $h_n(u)=0$ for $u>t_n^+$, we obtain
\begin{equation}\label{eq:dbct-affected}
\begin{aligned}
\E[J_n]&=\int_0^{t_n^+}\psi(h_n(u))A(du)\\
&\le\frac{c}{\sigma}(t_n^+)^\xi(2t_n\varepsilon_n)^\sigma\\
&\le Cb_n\varepsilon_n^\sigma.
\end{aligned}
\end{equation}
Consequently,
\[
\E\left[\frac{(\log n)^2J_n}{b_n}\right]\le C\varepsilon_n^\sigma(\log n)^2=Ct_n^{-\xi\sigma/4}(\log n)^2\longrightarrow0,
\]
and Markov's inequality gives
\[
\frac{(\log n)^2J_n}{b_n}\xrightarrow{\mathbb P}0.
\]
Combining this convergence with \eqref{eq:dbct-localization}, \eqref{eq:dbct-spectrum-change} and $\log(1+N(t_n^+))=O_p(\log n)$ yields
\[
\frac{\log n}{b_n}\sum_{r\ge1}|M_r(T_n)-M_r(t_n)|\log(1+r)\xrightarrow{\mathbb P}0.
\]

Let
\[
D_n=\sum_{r\ge1}\left|\frac{M_{r,n}}{b_n}-q_\sigma(r)\right|\log(1+r).
\]
By the triangle inequality,
\[
D_n\le D(t_n)+\frac{1}{b_n}\sum_{r\ge1}|M_r(T_n)-M_r(t_n)|\log(1+r).
\]
Together with \eqref{eq:dbct-deterministic} and $\log n=(1+\xi)\log t_n+\log a_0$, this gives
\[
(\log n)D_n\xrightarrow{\mathbb P}0.
\]

Finally, let $m_\sigma=\sum_{r\ge1}q_\sigma(r)\log(1+r)<\infty$. The normalization argument leading to \eqref{eq:scaled-normalization-bound} gives
\[
\sum_{r\ge1}\left|\frac{M_{r,n}}{K_n}-q_\sigma(r)\right|\log(1+r)\le\frac{b_n}{K_n}\left(1+\frac{m_\sigma}{\log2}\right)D_n.
\]
Since $b_n=C_0n^\gamma$, \eqref{eq:dbct-first-order} implies $b_n/K_n\to1$ almost surely. This proves \eqref{eq:dbct-weighted} and therefore Assumption $(B)$ is satisfied.

\section{Simulation study}\label{app:simulation}

We examine the behaviour of the Ewens--Pitman maximum likelihood estimator described in Theorem~\ref{thm:mle} via simulations. We consider four data generating processes which share the same limiting frequency spectrum $q_{\alpha^*}$, with $\alpha^*=0.5$, but differ in the growth of the number of blocks: by Theorem~\ref{thm:mle} we expect $\hat \alpha_n \to 0.5$, while the behaviour of $\hat \theta_n$ will depend on the specific scenario. 

The first two scenarios are designed to induce either a polynomially or a logarithmically divergent strength estimate. Let $m$ be a \emph{nominal} sample size: we will generate datasets with random sample sizes with $N_m$ such that $\E[N_m] = m$. To this end, set $\lambda_m=cm^\gamma\ell(m)$ and let $K_m\sim\operatorname{Poisson}(\lambda_m)$. Conditional on $K_m$, sample independent block sizes $X_{m,1},\ldots,X_{m,K_m}$ with distribution
\[
    \Pr(X_{m,j}=r\mid K_m)=\frac{q_{\alpha^*}(r)p_m^r}{H_{\alpha^*}(p_m)}, \qquad r\ge1,
\]
where $H_{\alpha^*}(p)=1-(1-p)^{\alpha^*}$ and $p_m\in(0,1)$ solves
\begin{equation}\label{eq:simulation-calibration}
\lambda_m\frac{{\alpha^*} p_m(1-p_m)^{\alpha^*-1}}{1-(1-p_m)^{\alpha^*}}=m,
\end{equation}
so that the random size $N_m=\sum_{j=1}^{K_m}X_{m,j}$ is such that $\mathbb E[N_m]=m$. We consider:
\begin{itemize}
    \item[(Scenario A)] $(\gamma,c)=(0.7,1)$ and $\ell(m)=1$;
    \item[(Scenario B)] $(\gamma,c)=(0.5,1)$ and $\ell(m)=\log m$.
\end{itemize}
In both scenarios, $N_m/m\to1$ and $K_m/\{N_m^\gamma\ell(N_m)\}\to c$ in probability. Hence, Theorem~\ref{thm:mle} implies, under Scenario A
\[
    \frac{\hat \theta_{N_m}}{N_m^{0.4}/4} \to 1 \text{ in probability, }
\]
while, under Scenario B,
\[
    \frac{\hat \theta_{N_m}}{\log^2(N_m)/4} \to 1 \text{ in probability. }
\]
Therefore, Scenarios A and B correspond to the first two regimes of  Theorem~\ref{thm:mle}, respectively.

Scenario C is a well-specified setting: we generate an Ewens--Pitman partition of size $N_m=m$ with parameters $(0.5,1.0)$. This corresponds to the third regime of Theorem~\ref{thm:mle}, which implies $\hat\theta_m=O_p(1)$.

Finally, scenario D considers $(\gamma,c)=(0.3,2)$ and $K_m=B_m+1$ with $B_m\sim\operatorname{Poisson}(cm^\gamma)$. The first block has size $m$, while the remaining $B_m$ block sizes are given by $X_{m,j} \overset{\text{i.i.d.}}{\sim} q_{\alpha^*}$, leading to $N_m=m+\sum_{j=1}^{B_m}X_{m,j}$. Intuitively, the proportion of observations belonging to the large cluster goes to $1$, but the number of clusters still grows as $n^\gamma$. This corresponds to the fourth regime of Theorem~\ref{thm:mle}, where $\hat\theta_{N_m}\to-0.5$ in probability.

We generate 200 independent partitions at each of 11 logarithmically spaced nominal sizes between $10^3$ and $10^8$. The results, reported in Figures~\ref{fig:simulation-strength} and~\ref{fig:simulation-discount}, largely agree with the theoretical predictions. The discount estimates approach the common limit $0.5$ (Figure \ref{fig:simulation-discount}), whereas the strength estimates behave differently across scenarios (Figure~\ref{fig:simulation-strength}). In Scenarios A and B,  the growth of $\hat \theta_m$ matches the predicted polynomial and logarithmic rates, respectively. In Scenario C, instead, the estimates remain on a bounded scale with persistent dispersion. In D, the median strength decreases towards $-0.5$, but convergence is slow: even at $m=10^8$, the estimates remain visibly separated from this limit, despite being negative.

\begin{figure}
\centering
\includegraphics[width=0.8\linewidth]{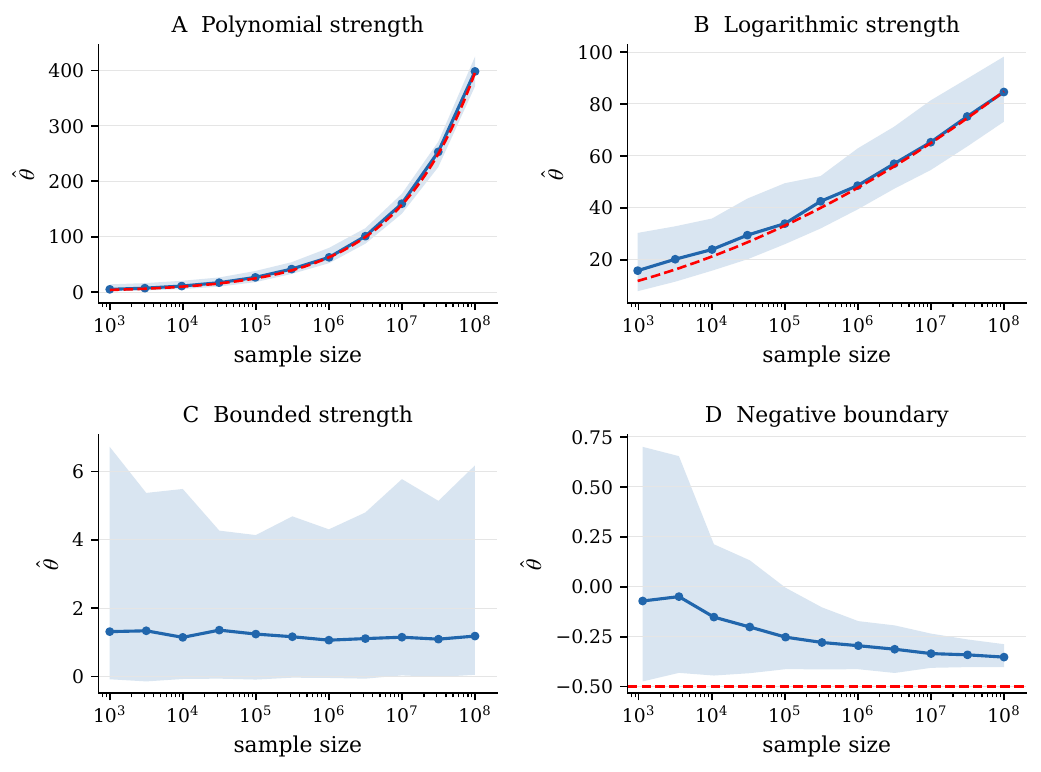}
\caption{Strength estimates in the four simulation designs. The blue line is the median over independent datasets, the blue area is a 95\% confidence interval, and the red shaded line is the target.}
\label{fig:simulation-strength}
\end{figure}

\begin{figure}
\centering
\includegraphics[width=0.8\linewidth]{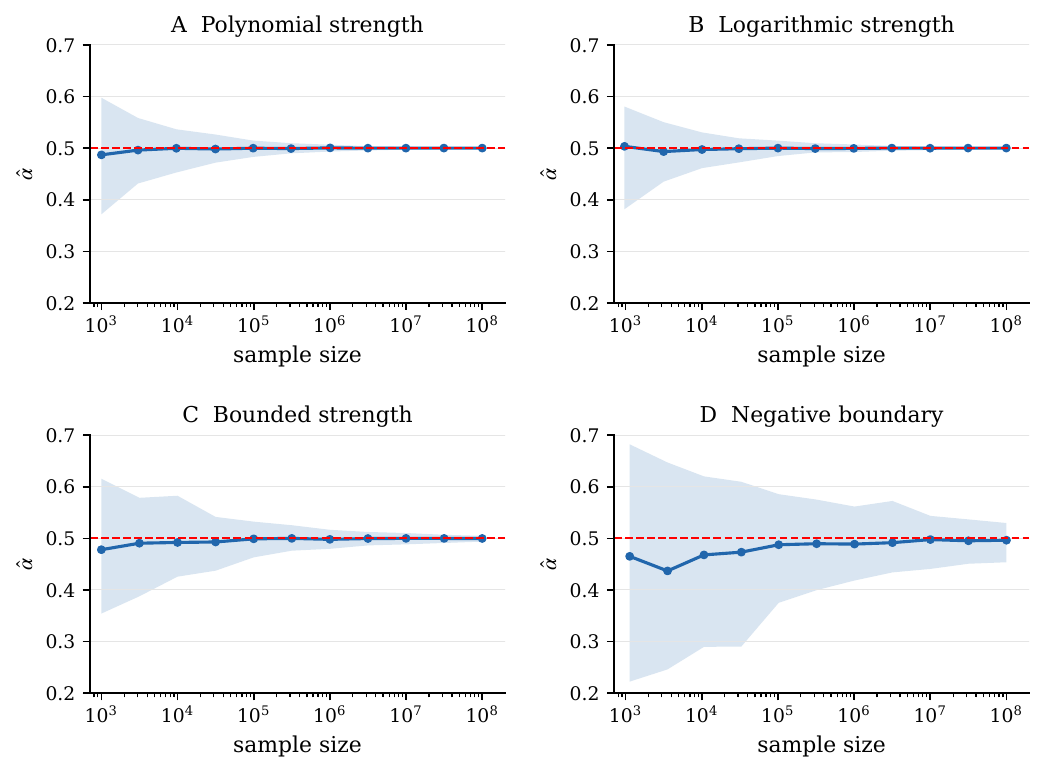}
\caption{Discount estimates in the four simulation designs. Red dashed lines mark the limit $\alpha^*=0.5$. All remaining details are as in the caption of Figure \ref{fig:simulation-strength}.}
\label{fig:simulation-discount}
\end{figure}

\end{document}